\documentclass[3p,12pt]{elsarticle}
\biboptions{numbers,sort&compress}

\usepackage{amsmath}
\usepackage{amssymb}
\usepackage{amsthm}
\usepackage{graphicx}
\usepackage{tikz}
\usetikzlibrary{arrows.meta,calc}
\usepackage{caption}
\usepackage{bookmark}
\usepackage{microtype}
\newcommand{\C}{\mathcal{C}}
\newcommand{\nC}{n_{\mathcal{C}}}
\newcommand{\taugrad}{\nabla_\tau}
\newcommand{\nderiv}{\partial_{\nC}}
\newcommand{\dA}{d_{\mathcal{A}}}

\newcounter{statement}[section]

\newtheorem{theorem}[statement]{Theorem}
\newtheorem{lemma}[statement]{Lemma}
\newtheorem{proposition}[statement]{Proposition}
\newtheorem{corollary}[statement]{Corollary}
\theoremstyle{definition}
\newtheorem{definition}[statement]{Definition}

\newtheorem{remark}[statement]{Remark}

\numberwithin{equation}{section}

\journal{Commun.\ Nonlinear Sci.\ Numer.\ Simul.}

\makeatletter
\@ifundefined{affiliation}{%
  \newcommand{\affiliation}[2][]{%
    \begingroup
    \def\organization##1{##1}\def\addressline##1{##1}%
    \def\city##1{##1}\def\postcode##1{##1}\def\country##1{##1}%
    \address[#1]{#2}%
    \endgroup}%
}{}
\makeatother

\begin{document}

\begin{frontmatter}

\title{Shortest Paths in Domains with Crack-Induced Constraints}

\author[fas,vnu]{Dau The Phiet\corref{cor1}}
\ead{dauthephiet@hcmut.edu.vn}
\cortext[cor1]{Corresponding author}

\author[imcs,fas,vnu]{Phan Thanh An}

\affiliation[imcs]{organization={Institute of Mathematical and Computational Sciences, Ho Chi Minh City University of Technology (HCMUT)},
            addressline={268 Ly Thuong Kiet, Dien Hong Ward},
            city={Ho Chi Minh City},
            country={Vietnam}}
\affiliation[fas]{organization={Faculty of Applied Science, Ho Chi Minh City University of Technology (HCMUT)},
            addressline={268 Ly Thuong Kiet, Dien Hong Ward},
            city={Ho Chi Minh City},
            country={Vietnam}}
\affiliation[vnu]{organization={Vietnam National University Ho Chi Minh City},
            addressline={Linh Xuan Ward},
            city={Ho Chi Minh City},
            country={Vietnam}}

\begin{abstract}
We formulate the shortest-path problem in a domain containing a crack as a Hamilton--Jacobi equation carrying a state constraint on the crack, which on the binding part of the crack takes the form of a variational inequality: the crack is a two-sided barrier that may be touched and followed but not traversed, so that it acts as a lower-dimensional state constraint, carried separately by each of its two lips, which admits tangential propagation but forbids crossing. We establish a comparison principle and uniqueness of the value function through trajectory-based arguments that circumvent the failure of the classical doubling technique of Crandall, Ishii, and Lions at the crack. The variational inequality induces a free boundary on the crack separating a sliding regime, where optimal trajectories run along the crack over a finite segment, from an illuminated regime, where the crack is reached transversally by direct rays; we give a complete geometric characterization of both regimes, including grazing (tangential) contact, the absence of reflection at regular crack points, and a diffraction law at crack tips, all derived from optimality alone. A simulation of the forward problem on an exactly solvable configuration reproduces the three regimes and the tip-diffraction law to the accuracy of the closed-form solution.
\end{abstract}


\end{frontmatter}

\section{Introduction}

The computation of distance functions and shortest paths via Hamilton--Jacobi equations has been extensively studied. Classical formulations based on the Eikonal equation admit a well-developed viscosity solution theory, with foundational contributions by Michael G. Crandall and Pierre-Louis Lions \cite{CrandallLions1983, CrandallIshiiLions1992, BarlesPerthame1988}; comprehensive expositions can be found in \cite{Barles1994, BardiCapuzzoDolcetta1997, CannarsaSinestrari2004, FlemingSoner2006}. In its simplest form, the distance function $u$ from a source point $x^*$ is characterized as the viscosity solution of
\[
|\nabla u(x)| = 1,
\]
posed in a domain $\Omega$ with appropriate boundary conditions. Efficient numerical methods such as the Fast Marching Method \cite{Sethian1999, KimmelSethian1998}, ordered upwind methods \cite{SethianVladimirsky2003}, and the Fast Sweeping Method \cite{Zhao2005} enable the computation of such solutions with optimal complexity and have become standard tools in scientific computing. For shortest-path computation in polygonal domains, a related discrete-geometry literature has developed independently \cite{MitchellMountPapadimitriou1987, AnHoai2012, AnPhu2023, AnHoaiThinh2024}.

Despite their success, classical Eikonal-based models implicitly assume that the domain is either homogeneous or contains obstacles treated as hard constraints (e.g., excluded regions or absorbing boundaries). However, many applications involve lower-dimensional geometric structures, such as cracks, fractures, or thin barriers, which fundamentally alter the topology of admissible paths. These structures do not merely block propagation; instead, they induce nontrivial geometric behaviors in optimal trajectories, including tangential contact and sliding along the constraint.

Existing approaches typically model cracks either by removing them from the domain or by imposing boundary conditions along their interface. While these strategies ensure well-posedness, they do not fully capture the intrinsic variational structure of shortest paths near cracks. In particular, they fail to describe the transition between trajectories that merely touch the crack and those that propagate along it. This transition is inherently a free-boundary phenomenon and lies beyond the scope of the standard Eikonal framework.

A related variational setting is the obstacle problem, where the solution is constrained to lie above a prescribed obstacle. This leads to a variational inequality with a complementary condition distinguishing free and contact regions \cite{KinderlehrerStampacchia1980, Friedman1982, CaffarelliSalsa2005}. Although this framework also involves a free boundary, the constraint acts directly on the function value rather than on its gradient or admissible directions. As a consequence, it does not capture the geometric behavior of shortest paths interacting with embedded constraints such as cracks.

In this work, we propose a Hamilton--Jacobi formulation with a variational inequality on the crack set that naturally encodes these effects. Specifically, we consider a domain $\Omega$ containing a crack $\C$, and define the distance function $u$ as the value function of an optimal control problem in which trajectories are constrained to avoid crossing the crack while being allowed to move tangentially along it. This is closely related to Hamilton--Jacobi equations with state constraints \cite{Soner1986, Soner1986II, Frankowska1993, IshiiKoike1996}. The resulting model consists of the classical Eikonal equation in the bulk together with a state-constraint condition on each lip of $\C$ that enforces both non-penetration and tangential propagation; where that condition binds it takes the form of a variational inequality, and Remark~\ref{rem:reconcile} delimits precisely where it does.

On the analytic side, the model is close in spirit to Hamilton--Jacobi equations posed on networks and stratified domains \cite{AchdouCamilliCutriTchou2013, BarlesBrianiChasseigne2013, ImbertMonneau2017}, in which the dynamics on the lower-dimensional set are coupled to the ambient equation through admissibility and matching conditions. 

This formulation reveals that the crack induces a free-boundary structure separating regions where the constraint is active (pure contact) from regions where propagation along the crack occurs (sliding). Optimal trajectories exhibit a rigid geometric behavior: at an \emph{interior} contact with a regular crack point the velocity is tangential, so that a trajectory passing a crack point either grazes it or slides along it; a crack point may still be reached transversally, but only as the endpoint of the trajectory. Moreover, we show that an optimal trajectory never reflects off the crack at a regular point, and that in the plane it changes direction only at crack tips, at junctions, and along sliding arcs, the behaviour at a tip being constrained by a one-sided diffraction inequality; all of this is derived purely from optimality.

Our main contributions are as follows:
\begin{itemize}
\item We introduce a Hamilton--Jacobi variational inequality framework for shortest paths in domains with cracks, which enforces non-penetration while allowing tangential propagation along the crack. The crack is treated as a two-sided barrier and the equation is posed on the slit completion of the domain, each of its two lips carrying its own state constraint; we show by example that a co-orientation of a crack making one lip systematically the cheaper one need not exist, so that the two-sided formulation is a necessity rather than a technical convenience.
\item We prove well-posedness of the model in the viscosity sense: the value function is the unique viscosity solution, and it is sandwiched between every subsolution and every supersolution (Theorems~\ref{thm:main} and~\ref{thm:comparison}).
\item We identify and characterize the free boundary on the crack that separates sliding and non-sliding regimes.
\item We describe optimal trajectories in the plane, for a convex domain and a $C^2$ crack (Theorem~\ref{thm:structure}): they consist of straight segments, sliding arcs along the crack, and pivots at tips and junctions; interior contact at regular crack points is tangential and never reflective, and at a tip the trajectory bends only around the crack, the tip acting as a secondary source.
\end{itemize}

The proposed framework bridges classical distance function computation with constrained variational problems on lower-dimensional structures. It also opens the way to extensions involving directional constraints, visibility, and anisotropic propagation, with potential applications in robotic sensing and path planning in complex environments.

\section{Problem Formulation}\label{sec:formulation}

Let $\Omega \subset \mathbb{R}^n$ be a bounded open domain with $C^{1,1}$ boundary, and let $\C \subset \overline{\Omega}$ denote a closed set representing a crack. We assume that $\C$ is a finite union of compact $C^{1,1}$ hypersurfaces-with-boundary, that distinct sheets meet transversally and in a set of dimension at most $n-2$, and that $\C$ meets $\partial\Omega$ transversally. Transversality excludes cusps: it guarantees that each local component of $(B_r(x)\cap\Omega)\setminus\C$ has a tangent cone with non-empty interior at $x$, equal to the closure of that interior. At a corner of a merely Lipschitz boundary this already fails, the tangent cone being strictly larger than the set of directions along which one can move at all; this is why $\partial\Omega$ is taken smooth. Being $C^{1,1}$ and bounded, $\overline\Omega$ is quasiconvex. Theorem~\ref{thm:structure} will require more, namely $C^2$ arcs in the plane, and says so. At each regular point $x \in \C$ we denote by $T_x \C$ the tangent space and by $\nu(x)$ a unit normal, defined only up to sign.

\begin{remark}[Why $C^{1,1}$ and not $C^1$]\label{rem:c11}
Two of the arguments below need a \emph{uniform} two-sided tubular neighborhood of each regular point, and a merely $C^1$ sheet need not have one, because its curvature may be unbounded and its cut locus may then reach the sheet itself. For a witness, take $n=2$, $\Omega=B(0,2)$ and $\C$ the graph of $f(t)=|t|^{3/2}$ over $[-1,1]$, which is a compact connected $C^1$ arc. At the regular point $y=0$ the concave side carries a cut locus along the whole segment $\{(0,t):0<t<t_0\}$: the point $(0,t)$ has the two nearest points $(\pm s_*,s_*^{3/2})$ with $s_*=\tfrac94t^2+o(t^2)$, and $\operatorname{dist}(\cdot,\C)$ has a genuine corner there --- its one-sided derivatives in $t_1$ at $(0,t)$ are $\mp0.2131$ for $t=0.1$ and $\mp0.00267$ for $t=0.001$, computed directly. So the cut locus accumulates at $y$, no neighborhood of $y$ carries a differentiable signed distance, and the set over which clause~(ii) of Definition~\ref{def:admissible} quantifies at $y$ would be empty --- leaving a transversal crossing at $y$ admissible and collapsing $\dA$ between the two lips over $y$ to zero. Under $C^{1,1}$ the second fundamental form is bounded, the reach of each sheet is positive, and both the tubular neighborhood and the normal push-off used in Remark~\ref{rem:lift} are available. We use no more than this.
\end{remark}
The crack is a \emph{two-sided barrier}: a trajectory may touch it and travel along it, but may not pass from one of its sides to the other.
No global co-orientation of $\C$ is chosen, and none is available in general.
Instead, a trajectory meets the crack from one of the two sides, and we write $\nC$ for the unit normal pointing \emph{away} from that side, that is, toward the side the trajectory may not enter.
This convention is local and involves no choice: the two sides are kept apart by working in the slit completion $\widehat\Omega$ of Remark~\ref{rem:slit}, in which every regular crack point is split into two lips, one for each side, and $\nC$ is single-valued on each lip.
With it, the velocities available at a crack point satisfy
\[
d\cdot \nC\le 0,
\]
which prevents trajectories from crossing the crack.

It is convenient to record the constraint through a signed distance function.
For a regular point $y\in\C$, let $U$ be a sufficiently thin tubular neighborhood of $y$ meeting no tip or junction; such a $U$ exists, and is disjoint from the cut locus, because the sheet carrying $y$ is $C^{1,1}$ and hence of positive reach (Remark~\ref{rem:c11}). On $U$ the \emph{signed distance to the crack}
\[
s_\C(x):=\pm\operatorname{dist}(x,\C),
\]
with the sign $-$ on the side from which the crack is approached and $+$ on the far side, is of class $C^{1,1}$ and satisfies $|\nabla s_\C|\equiv1$ in $U$ together with $\nabla s_\C(y)=n_\C(y)$ for $y\in\C$. Thus the side carrying the trajectory is locally $\{s_\C<0\}$, and the velocity constraint $d\cdot\nC\le0$ is the same as $d\cdot\nabla s_\C\le0$. Only the sign of $s_\C$ depends on the side, and that is immaterial here: the non-crossing requirement below is invariant under changing the sign, so no globally defined branch of $s_\C$ is needed. This matters because none exists in general: around a tip the two sides merge, and the sign flips as one encircles it.
Moreover, even where the sign is fixed, $s_\C$ is smooth only away from the cut locus of $\C$, where $\operatorname{dist}(\cdot,\C)$ ceases to be differentiable. Under $C^{1,1}$ regularity this never intervenes: each sheet has positive reach, so every regular point has a tubular neighborhood $U$ disjoint from the tips, junctions, and cut locus, which is all that is used below. Remark~\ref{rem:c11} shows that this is exactly what $C^1$ fails to provide.

For brevity we call $\{s_\C<0\}\cap U$ the \emph{near side} of the crack at $y$ --- the side from which $y$ is approached, equivalently the side of the lip under consideration --- and $\{s_\C>0\}\cap U$ the \emph{far side}. These two names are attached to a lip, not to the crack: the near side of one lip is the far side of the other, and no global labelling of the two sides of $\C$ is implied by them.

It is convenient to record, once and for all, the following structure and terminology.
Write $\C=\bigcup_j\C_j$ as a finite union of compact connected $C^1$ hypersurfaces-with-boundary $\C_j\subset\overline\Omega$.
The \emph{junction set} $\C_{\mathrm{self}}$ consists of the points lying on two or more sheets (junctions where sheets meet, and crossings) together with the relative boundary points of the sheets that lie on $\partial\Omega$; it is lower-dimensional.
The \emph{tip set} $\partial\C:=\bigl(\bigcup_j\partial\C_j\bigr)\setminus\C_{\mathrm{self}}$ consists of the relative (manifold) boundary points that belong to a single sheet and lie in the interior of $\Omega$, that is, the free ends of the crack; it is $(n-2)$-dimensional and reduces to finitely many points when $n=2$.
Only these can be rounded: where a sheet ends on $\partial\Omega$ there is no room to pass on the far side within $\overline\Omega$, so such an endpoint is sealed and is counted as a junction, where crossing is forbidden by condition~(ii) of Definition~\ref{def:admissible}.
We write $\C_{\mathrm{sing}}:=\partial\C\cup\C_{\mathrm{self}}$ for the singular set (tips and junctions) and $\C_{\mathrm{reg}}:=\C\setminus\C_{\mathrm{sing}}$ for the regular part.
Here $\partial\C$ is the \emph{relative} boundary of the hypersurface, not the topological boundary of $\C$ in $\mathbb{R}^n$ --- the latter, $\C$ having empty interior, would be all of $\C$.
We further take the crack to obstruct genuinely: at every $x\in\C_{\mathrm{reg}}$ both half-neighborhoods $\{s_\C>0\}\cap\Omega$ and $\{s_\C<0\}\cap\Omega$ are non-empty.
This is in fact automatic: by the definition of $\C_{\mathrm{self}}$ just given, together with the transversality of $\C$ and $\partial\Omega$, every point of $\C\cap\partial\Omega$ is a junction, so $\C_{\mathrm{reg}}\subset\Omega$. We record the property because the half-sphere formulas for $A$ and $A^{\mathrm{arr}}$, the computation \eqref{eq:sup-explicit}, and the step ``both $\tau$ and $-\tau$ belong to $A^{\mathrm{arr}}(x_0)$'' in Step~1 below all presuppose it, and all would fail at a regular crack point lying on $\partial\Omega$.
Finally, near a tip $p\in\partial\C$ the two sides join --- $\Omega\setminus\C$ is connected in a punctured neighborhood of $p$ --- which is exactly what allows an admissible trajectory to pass from one side of a sheet to the other by rounding $p$.
At a junction $y\in\C_{\mathrm{self}}$, by contrast, the sheets meeting at $y$ may seal the passage, and the relevant object is the set of connected components of $(B_r(y)\cap\Omega)\setminus\C$ for small $r>0$: a trajectory may not pass from one component to another, which is the junction clause of Definition~\ref{def:admissible}.
No labelling of the components is needed.

Let $x^* \in \Omega \setminus \C$ be a fixed source point. We consider curves $\gamma : [0,T] \to \overline{\Omega}$ satisfying
\[
\gamma(0) = x^*, \quad \gamma(T) = x,
\]
which are constrained by the presence of the crack. Specifically, trajectories are not allowed to cross $\C$, but may move tangentially along it.

This constraint is encoded through the admissible velocity set. A trajectory sitting at $x$ may move only into the part of the domain available to it there, which --- once the two sides of the crack are kept apart --- is one connected piece of a small punctured neighborhood. Write $K(\hat x)$ for that piece: for $\hat x$ a lip over $x\in\C_{\mathrm{reg}}$ it is the near side of the lip, for a tip it is the whole punctured neighborhood, since the two sides merge there, and for a junction copy it is the corresponding connected component of $(B_r(x)\cap\Omega)\setminus\C$. Let $T(\hat x)$ denote the tangent cone of $\overline{K(\hat x)}$ at $x$: a closed cone which, by the transversality assumed above, has non-empty interior and is the closure of it. We set
\begin{equation}\label{eq:A}
A(\hat x)=\bigl\{d:|d|=1,\ d\in T(\hat x)\bigr\} .
\end{equation}
Explicitly, $A(\hat x)$ is the full unit sphere at an interior bulk point and at a tip, the inward half-sphere at a point of $\partial\Omega\setminus\C$ --- the domain boundary is a constraint like any other, and the single formula records it --- and the closed half-sphere
\[
\{d:|d|=1,\ d\cdot n_\C(x)\le0\}
\]
on a lip over $x\in\C_{\mathrm{reg}}$ --- the half pointing into the side currently carrying the trajectory, which prevents crossing while still allowing tangential propagation. At a junction it is the set of unit vectors in the tangent cone of the component the trajectory occupies: at the corner of a polygonal crack, for instance, the directions pointing into that corner's sector. Writing the three cases by one formula matters at junctions, where the full sphere would permit a trajectory to leave a sector through a sheet. The tip is the one place where no restriction is imposed, and legitimately so: the two sides genuinely merge there and trajectories may round the tip freely.

For the backward (dynamic programming) analysis we also need the
\emph{arrival direction set}
\begin{equation}\label{eq:Aarr}
A^{\mathrm{arr}}(\hat x)=-A(\hat x)=\bigl\{d:|d|=1,\ -d\in T(\hat x)\bigr\},
\end{equation}
consisting of the terminal velocities with which $\hat x$ can be reached along an admissible curve: the predecessor $x-hd$ must itself be available, i.e. $-d$ must point into $K(\hat x)$. On a lip this reads
\[
A^{\mathrm{arr}}(\hat x)=\{d:|d|=1,\ d\cdot n_\C(x)\ge0\} ,
\]
since the predecessor lies on the near side, $s_\C(x-hd)\le0$. Thus the forward (departure) constraint $d\cdot\nC\le0$ forbids leaving the side, whereas the arrival constraint $d\cdot\nC\ge0$ records that one must have come \emph{from} that side: the two half-spheres are opposite. The distinction is invisible in the bulk and at tips, where both sets are the full sphere, but essential on $\C_{\mathrm{reg}}$, and it is exactly what the free-boundary analysis of Section~\ref{sec:freeboundary} renders transparent.

\begin{definition}[Admissible trajectory]\label{def:admissible}
A Lipschitz curve $\gamma:[0,T]\to\overline\Omega$ is \emph{admissible} if
\begin{itemize}
\item[(i)] $\dot\gamma(t)\in T(\hat\gamma(t))$ for a.e. $t\in[0,T]$, where $\hat\gamma$ is a lift of $\gamma$ supplied by~(ii) --- the two lips over a regular crack point carry opposite normals, so the constraint is attached to the lift and not to the projection, and (i) is to be read after (ii). We impose membership in the \emph{cone} $T(\hat\gamma)$ rather than in $A(\hat\gamma)=T(\hat\gamma)\cap S^{n-1}$, so that curves are not forced to have unit speed; this matters because the length functional below is minimized over curves of arbitrary parametrization. And
\item[(ii)] $\gamma$ does not cross $\C$: for every regular point $y\in\C_{\mathrm{reg}}$ and every tubular neighborhood $U$ of $y$ on which the local signed distance $s_\C$ is defined as above, the function $t\mapsto s_\C(\gamma(t))$ does not change sign on any subinterval where $\gamma$ remains in $U$; and for every junction $y\in\C_{\mathrm{self}}$ and every small ball $B=B_r(y)$, $\gamma$ stays in the closure of a single connected component of $(B\cap\Omega)\setminus\C$ on any subinterval where it remains in $B$. The intersection with $\Omega$ matters where the junction lies on $\partial\Omega$: there $B\setminus\C$ may be connected through the exterior of $\Omega$, which no trajectory may use, while $(B\cap\Omega)\setminus\C$ is not.
Both clauses are invariant under reversal of $\gamma$ and under the choice of sign of $s_\C$.
\end{itemize}
\end{definition}

\begin{remark}
Condition (ii) is not implied by (i): a trajectory crossing $\C$ transversally meets the crack on a set of times of measure zero, which the a.e. condition (i) cannot detect. Condition (ii) closes this loophole and encodes the actual geometric meaning of the crack as a barrier. Its junction clause closes a second loophole: at a corner of a polygonal crack no tubular neighborhood of a regular point is available, and without this clause a trajectory could slip through the crack exactly at the corner.
\end{remark}

We define the value function by
\[
u(\hat x) = \inf_{\gamma} \int_0^T |\dot{\gamma}(t)|dt,
\]
the infimum being taken over all admissible curves in the sense of Definition~\ref{def:admissible} from $x^*$ to $\hat x$. The argument is a point of the slit completion $\widehat\Omega$, and this matters: an admissible curve ends on one lip, so $u$ is \emph{two-valued} over a regular crack point, one value per lip, and it is these two values that Remark~\ref{rem:slit} calls the traces $u_\pm$. Where a single-valued function on $\overline\Omega$ is wanted --- as in the definition of the forward problem --- one takes the smaller trace, $\min(u_+,u_-)$; everything below works with the two-valued object on $\widehat\Omega$. By Remark~\ref{rem:lift}, $u=\operatorname{dist}_{\widehat\Omega}(x^*,\cdot)$.

By construction, $u(x)$ represents the shortest path distance from $x^*$ to $x$ under the constraint induced by the crack.

\begin{definition}[Admissible path quasi-metric]\label{def:pathmetric}
For $x,y\in\overline\Omega$, let
\[
\dA(x,y):=\inf\Bigl\{\textstyle\int_0^T|\dot\gamma|dt :
\gamma \text{ admissible},\gamma(0)=x,\gamma(T)=y\Bigr\},
\]
with the convention $\inf\emptyset=+\infty$. Admissibility is invariant under reversal of a curve: clause~(ii) is manifestly so, and clause~(i) follows, since by Remark~\ref{rem:lift} it is implied by clause~(ii) at almost every time. Hence the reverse of an admissible curve is admissible and $\dA$ is symmetric. It is a metric on the slit completion $\widehat\Omega$ of Remark~\ref{rem:slit}, with values in $[0,+\infty]$, the value $+\infty$ occurring exactly between distinct connected components, as happens when the crack closes around one of the two points. It is \emph{not} a metric on $\overline\Omega$: the triangle inequality fails at every regular crack point, since two curves ending on opposite lips of $y\in\C_{\mathrm{reg}}$ cannot be concatenated at $y$ without violating clause~(ii), so that $\dA(x,z)$ may exceed $\dA(x,y)+\dA(y,z)$ by the cost of rounding a tip. All metric statements below are therefore made on $\widehat\Omega$, and we write $\dA(x,y)$ for points of $\overline\Omega$ only when they are off $\C_{\mathrm{reg}}$ or when a lip has been specified.
\end{definition}

\begin{remark}[Continuity across the crack]\label{rem:slit}
The value function satisfies $u(y)\le u(x)+\dA(x,y)$, hence it is $1$-Lipschitz with respect to $\dA$. It is, however, discontinuous across the crack in the Euclidean topology in general: the two sides of a regular crack point can be joined only by paths circumventing the tips, so the two one-sided traces may differ, and do in general. They need not: for $\C=[-1,1]\times\{0\}$ in $B(0,3)$ with $x^*=(2,0)$ the two traces both equal $2-t$ at $(t,0)$, so $u$ happens to be Euclidean-continuous across that crack. Neither trace dominates the other in general; which of the two is the smaller may even change along a single sheet, as the example of Remark~\ref{rem:lips} shows, and the set where they agree is the locus at which both faces are reached at the same cost. The natural topology for the problem is therefore that of the completion $\widehat\Omega$ of $\Omega\setminus\C$ with respect to the length metric (the classical \emph{slit-domain completion}, in which each regular crack point is doubled into two lips $\C^\pm$, one for each side ($s_\C\to0^\mp$); for each regular crack point $x\in\C_{\mathrm{reg}}$ the two one-sided traces are well-defined (since $u$ is $1$-Lipschitz with respect to $\dA$, and two points in the same half $U\cap\{s_\C<0\}$ or $U\cap\{s_\C>0\}$ of a small tubular neighborhood $U$ of $x$ can be joined inside that half by a path of length comparable to their distance, $u$ is Lipschitz, hence uniformly continuous, on each half): $u_\pm(x):=\lim_{y\to x,\,\pm s_\C(y)<0}u(y)$, one for each lip. At a tip $p\in\partial\C$ the two sides merge and only a single trace exists; there $u$ is continuous in the Euclidean topology). All continuity and semicontinuity statements below are understood in $\widehat\Omega$; on $\overline\Omega\setminus\C$ the two topologies agree.
\end{remark}

\begin{remark}[Admissibility is liftability]\label{rem:lift}
Condition~(ii) of Definition~\ref{def:admissible} has a purely topological reading: a Lipschitz curve $\gamma:[0,T]\to\overline\Omega$ satisfies it if and only if it admits a continuous lift $\hat\gamma:[0,T]\to\widehat\Omega$ with $\pi\circ\hat\gamma=\gamma$, where $\pi:\widehat\Omega\to\overline\Omega$ is the natural projection.
Indeed, a curve that touches a regular point from one side and leaves to the same side, or that runs along $\C$, lifts continuously to the corresponding lip, whereas a transversal crossing makes the lift jump between the two lips; a curve rounding a tip lifts continuously, the two lips meeting there in $\widehat\Omega$; and at a junction the components of $(B_r(y)\cap\Omega)\setminus\C$ are separated in $\widehat\Omega$, so remaining in the closure of one of them is exactly continuity of the lift.
Condition~(i) is then automatic at every \emph{interior} time of differentiability, rather than merely almost everywhere. Liftability places $\gamma$ in $\overline{K(\hat\gamma(t))}$ on a whole interval around $t$, so for $t\in(0,T)$ every difference quotient $(\gamma(t+h)-\gamma(t))/h$, of either sign of $h$, points into that set, and the limit $\dot\gamma(t)$ lies in its tangent cone $T(\hat\gamma(t))$ --- indeed in $T\cap(-T)$, which at a lip forces tangency, the statement Theorem~\ref{thm:tangential} will recover from optimality alone. At the two endpoints only one sign of $h$ is available and the conclusion is correspondingly one-sided: at $t=T$ one gets $\dot\gamma(T)\in-T(\hat\gamma(T))$, i.e.\ $\dot\gamma(T)\in\mathbb R_+A^{\mathrm{arr}}(\hat\gamma(T))$, and~(i) itself may fail there. It must: the trajectory $t\mapsto(0,1-t)$ with $\C=[-1,1]\times\{0\}$ is admissible and reaches the crack at $t=1$ with $\dot\gamma(1)=-n_\C\notin T$, which is exactly the transversal arrival at an endpoint announced in the introduction. This is why condition~(i) of Definition~\ref{def:admissible} is imposed almost everywhere and why the dynamic programming principle is phrased through $A^{\mathrm{arr}}$ and not $A$. The argument also covers the times spent on $\partial\Omega$ and on $\C_{\mathrm{sing}}$, which an argument through the signed distance would not, and it is where it matters that~(i) is stated with the cone rather than with $A$: a liftable curve may pause, and $\dot\gamma=0$ lies in every tangent cone but in no $A(\hat x)$. Admissibility and liftability therefore coincide.

It follows that $\dA$ is the length metric of $\widehat\Omega$, but this deserves a word, being the hinge of everything below. By definition $\dA$ is an infimum over admissible curves in $\overline\Omega$, which may \emph{run along} $\C$ and $\partial\Omega$, whereas the length metric of the open set $\Omega\setminus\C$ is an infimum over curves avoiding them and is a priori strictly larger; the completion only adds limit points. The two agree because an admissible curve can be pushed off the obstacle at vanishing cost in length: for a sub-arc running along a lip, the push-off $\gamma+\varepsilon\,n_\C\circ\gamma$ lies in $\Omega\setminus\C$ for small $\varepsilon$ and is Lipschitz with length $(1+O(\varepsilon))\,\mathrm{length}(\gamma)$, because $n_\C$ is Lipschitz on a $C^{1,1}$ sheet; near a tip or a junction one first pushes radially off the singular set, whose contribution to the length is $O(\varepsilon)$ by transversality, and similarly along $\partial\Omega$. This is a second place where $C^1$ would not do: for $\C$ the graph of a function with $f'(t)=t\sin(1/t)$ --- a $C^1$ arc --- the normal $n_\C$ has unbounded variation near $0$ and $\gamma+\varepsilon n_\C$ need not even be rectifiable.

Two properties of this metric space are used repeatedly and are recorded here. First, each connected component of $\widehat\Omega$ is complete and totally bounded, hence a compact length space. It is complete by construction, being a completion, and it is a length space because $\dA$ is defined as an infimum of lengths. For total boundedness, note that $\overline\Omega$ is a bounded quasiconvex $C^{1,1}$ domain and $\C$ is a finite union of compact sheets, so $\dA$ is bounded on each component and, away from a tubular neighborhood of $\C$ of radius $\rho$, is comparable to the Euclidean metric; a finite $\varepsilon$-net off that neighborhood together with a finite net on each lip --- the lips being totally bounded in the intrinsic metric of the sheets, which is all that is needed and which holds since $\C_{\mathrm{reg}}$, though not closed, is a subset of the compact set $\C$ --- covers the component. We stress that $\dA$ is \emph{not} comparable to the Euclidean metric up to $\C$: two points on opposite lips may be Euclidean-close and $\dA$-far, which is precisely the phenomenon exploited in Lemma~\ref{lem:DPP}. Second, $\dA$ takes the value $+\infty$ exactly between distinct components, so all metric statements are to be read componentwise. In particular each component is complete and locally compact, and the Hopf--Rinow--Cohn-Vossen theorem applies to it: any two of its points are joined by a minimizing geodesic. With this,
\[
u=\operatorname{dist}_{\widehat\Omega}(x^*,\cdot).
\]
This identification is what makes the two-sided model well posed without any labelling of the two sides of $\C$, and we use it repeatedly below. It should be said plainly that the identification itself is elementary; the substance of the Hamilton--Jacobi formulation lies in the local structure it forces on each lip --- the free boundary, the tangency and no-reflection laws, the tip condition --- which is the subject of Section~\ref{sec:freeboundary}.
\end{remark}

\medskip

In the absence of the crack, the function $u$ is characterized as the unique viscosity solution of the Eikonal equation
\[
|\nabla u| = 1.
\]
In the present setting, the crack imposes directional constraints that modify this characterization. The value function $u$ satisfies a Hamilton--Jacobi equation in the bulk together with a variational inequality on $\C$.

More precisely, we consider the following system on the slit completion:
\begin{equation}
\begin{cases}
|\nabla u(x)| = 1 & \text{in } (\Omega \setminus \C)\cup\partial\C, \\[6pt]
\displaystyle\sup_{d\in A^{\mathrm{arr}}(x)} d\cdot\nabla u(x) = 1 & \text{on } \partial\Omega\setminus\C, \\[6pt]
\displaystyle\sup_{d\in A^{\mathrm{arr}}(x)} d\cdot\nabla u(x) = 1 & \text{on each lip and each junction copy,} \\[6pt]
u(x^*) = 0,
\end{cases}
\label{eq:HJVI}
\end{equation}
where $A^{\mathrm{arr}}$ is the arrival set \eqref{eq:Aarr}; the three lines above the source condition are one and the same requirement $\sup_{d\in A^{\mathrm{arr}}}d\cdot\nabla u=1$, which at an interior bulk point or a tip reduces to $|\nabla u|=1$, the arrival set being the full unit sphere there, and which on $\partial\Omega$ and at a junction copy is the supremum over the arrival directions available inside the component occupied. The domain boundary is thus a state constraint of the same kind as a lip, and is treated as one throughout. Thus \eqref{eq:HJVI} prescribes a condition at every point of $\widehat\Omega$, and it is in this single-line form that Definitions~\ref{def:sub} and~\ref{def:super} render it.

In \eqref{eq:VI} below, $\nabla_\tau u$ denotes the tangential gradient along $\C$ and $\partial_{\nC} u$ the normal derivative on a lip.
The lip condition in \eqref{eq:HJVI} is the state-constrained form of the eikonal equation: it encodes both the non-penetration constraint and the possibility of tangential propagation.
Where the constraint binds it takes the form of a variational inequality,
\begin{equation}\label{eq:VI}
\max\big(|\nabla_\tau u(x)| - 1,\ -\partial_{\nC} u(x)\big) = 0 ,
\end{equation}
whose two arguments are the tangential speed and the normal derivative.
A caution is in order about this form, and we give it here rather than let it mislead. For the value function the two arguments vanish \emph{together}: by Lemma~\ref{lem:normalsign}, $|\taugrad u|^2+(\nderiv u)^2=1$ at every lip point of differentiability, so $|\taugrad u|=1$ if and only if $\nderiv u=0$. The complementarity in \eqref{eq:VI} is therefore degenerate --- it is not a genuine either/or --- and for $u$ it reduces to the single scalar statement that the crack is reached tangentially. The information is carried by the state-constraint condition on the lip, of which \eqref{eq:VI} is only the binding-regime shadow. That \eqref{eq:VI} holds exactly on the binding part of each lip, and fails as an equation where the crack is reached transversally, is the content of Section~\ref{sec:freeboundary}.

This formulation naturally leads to the emergence of a free boundary on the crack, separating regions where trajectories only touch the crack from those where they propagate along it. The analysis of this structure and its implications for optimal trajectories is the main focus of this work.

The crack is thus an obstacle that may be touched and followed but not traversed, and the pair of lips at a regular point records the two sides separately.
Whether one lip is cheaper to reach than the other is a property of the configuration, not a modelling choice, and it is this that the free-boundary analysis of Section~\ref{sec:freeboundary} describes.

\section{Viscosity Solutions}\label{sec:viscosity}

In this section, we define viscosity solutions for the Hamilton--Jacobi variational inequality \eqref{eq:HJVI}.

\subsection{Preliminaries}

For $x \in \C$, we denote by $T_x \C$ the tangent space to $\C$ at $x$, and define the tangential gradient of a smooth function $\varphi$ by
\[
\nabla_\tau \varphi(x) = \nabla \varphi(x) - (\nabla \varphi(x)\cdot n_\C(x)) n_\C(x),
\]
where $n_\C(x)$ is a unit normal vector to $\C$. The normal derivative is given by
\[
\partial_{n_\C} \varphi(x) = \nabla \varphi(x)\cdot n_\C(x).
\]

Each $C^1$ sheet of $\C$ carries the Riemannian metric induced by the Euclidean structure of $\mathbb R^n$. For a function $f\in C^1(\C)$ living on the crack we write $\nabla_{\!\C} f$ for its \emph{intrinsic gradient} (the Riemannian gradient of the induced metric, a tangent vector in $T_x\C$), and we call the associated geodesic distance the \emph{intrinsic distance} on $\C$; a curve in $\C$ that is locally length-minimising for this metric is an \emph{intrinsic geodesic}. When $f=\varphi|_\C$ is the trace of an ambient $C^1$ function $\varphi$, the intrinsic gradient coincides with the tangential gradient defined above,
\[
\nabla_{\!\C}(\varphi|_\C)=\nabla_\tau\varphi,
\qquad\text{so in particular}\qquad
|\nabla_{\!\C}(\varphi|_\C)|=|\nabla_\tau\varphi| .
\]
Accordingly, the \emph{intrinsic eikonal equation} on $\C$ is $|\nabla_{\!\C} f|=1$, and, \emph{for a locally Lipschitz $f$}, being $1$-Lipschitz for the intrinsic distance is equivalent to $|\nabla_{\!\C}f|\le1$ at every point of differentiability. The Lipschitz hypothesis cannot be dropped in the direction from the gradient bound to the Lipschitz bound --- the Cantor function has vanishing derivative wherever it is differentiable and is not Lipschitz --- and it is satisfied by $u|_\C$, locally: an arc contained in a single connected component of $\C_{\mathrm{reg}}$ is itself an admissible curve along either lip, so $u$ restricted to a lip is $1$-Lipschitz for the intrinsic distance of \emph{that component}, as soon as $u$ is $1$-Lipschitz for $\dA$, which it is by concatenation of admissible curves. The restriction to one component is not pedantry: for the cross $\C=\bigl([-1,1]\times\{0\}\bigr)\cup\bigl(\{0\}\times[-1,1]\bigr)$ in $B(0,2)$, the points $a=(-\tfrac12,0)$ and $b=(\tfrac12,0)$ on the \emph{same} upper lip of the horizontal bar are at intrinsic distance $1$ in $\C$, whereas $\dA(a,b)=2\sqrt{5}/2=2.2361$, the junction at the origin forcing a detour around the tip $(0,1)$; with $x^*=(-\tfrac12,\tfrac1{10})$ one gets $|u(a)-u(b)|=2.0476$, so the global claim is false by a factor of two. Since differentiability is a local matter, the componentwise statement is all that is used; where $u$ is differentiable up to a lip the bound $|\taugrad u|\le1$ also follows directly from $|\nabla u|=1$ (Lemma~\ref{lem:normalsign}).

\subsection{Definition of viscosity solutions}

Because the crack is two-sided, the two lips over a regular crack point have to be tested separately, and we therefore pose the viscosity theory on the slit completion $\widehat\Omega$ of Remark~\ref{rem:slit} rather than on $\overline\Omega$.
Let $\pi:\widehat\Omega\to\overline\Omega$ be the natural projection; it is a homeomorphism over $\overline\Omega\setminus\C$, two-to-one over $\C_{\mathrm{reg}}$, one-to-one over $\partial\C$, and $k$-to-one over a junction around which $\Omega\setminus\C$ has $k$ local components.
Test functions remain ambient, $\varphi\in C^1(\overline\Omega)$, and are compared with $u$ through $\pi$: a local extremum of $u-\varphi\circ\pi$ is taken in $\widehat\Omega$, so that at a lip only the values of $u$ on the side of that lip enter the comparison.
At a lip $\hat x_0$ lying over $x_0\in\C_{\mathrm{reg}}$ we write $\nabla\varphi(\hat x_0):=\nabla\varphi(x_0)$, we let $\nC(\hat x_0)$ be the unit normal pointing to the far side of that lip, and we form $A^{\mathrm{arr}}(\hat x_0)$ from this normal as in~\eqref{eq:Aarr}.
Over $\overline\Omega\setminus\C$ nothing changes, and we drop the hats whenever no confusion can arise; this convention is in force for the whole of Section~\ref{sec:viscosity} and Section~\ref{sec:comparison}.
Posing everything on $\widehat\Omega$ keeps the two lips genuinely independent and removes any need to collapse them into a single function on $\overline\Omega$; in particular the comparison lemmas of Section~\ref{sec:comparison} are stated and proved on $\widehat\Omega$ throughout.

\begin{definition}[Viscosity subsolution]\label{def:sub}
A function $u\in\mathrm{USC}(\widehat\Omega)$ is a \emph{viscosity subsolution} of~\eqref{eq:HJVI} if for every test function $\varphi\in C^1(\overline\Omega)$ and every point $x_0\in\widehat\Omega\setminus\{x^*\}$ with $u(x_0)<\infty$ at which $u-\varphi$ attains a local maximum:
\begin{equation}\label{eq:sub-bellman}
\sup_{d\in A^{\mathrm{arr}}(x_0)} d\cdot\nabla\varphi(x_0)\ \le\ 1 .
\end{equation}
The two provisos match those of Definition~\ref{def:super} and are both needed. Finiteness: on a component of $\widehat\Omega$ not containing $x^*$ the value function is identically $+\infty$, so $u-\varphi\equiv+\infty$ there and \emph{every} point of that component is a local maximum; without the proviso $u$ would fail to be a subsolution whenever $\widehat\Omega$ is disconnected, as it is for the closed sheet $\C=\{|x|=1\}$ in $B(0,3)$ with $x^*=(0,2)$, where $\varphi(x)=2x_1$ and $x_0=0$ would demand $2\le1$. The source: at $x^*$ the condition is in any case vacuous, since $\dA(x^*,y)\le\varphi(y)-\varphi(x^*)$ near $x^*$ would force $d\cdot\nabla\varphi(x^*)\ge1$ for $d$ and $-d$ alike, so $u-\varphi$ has no local maximum there.
At a bulk point and at a tip the arrival set is the full unit sphere and \eqref{eq:sub-bellman} reads $|\nabla\varphi(x_0)|\le1$; at a junction copy it is the set of unit $d$ with $-d$ in the tangent cone of the component occupied; on a lip of $\C_{\mathrm{reg}}$ it reads $|\nabla\varphi(x_0)|\le1$ when $\nderiv\varphi(x_0)\ge0$ and $|\taugrad\varphi(x_0)|\le1$ when $\nderiv\varphi(x_0)\le0$, by the computation of Remark~\ref{rem:strongsuper}.
We emphasise that the stronger max-form inequality $\max(|\taugrad\varphi|-1,-\nderiv\varphi)\le0$ is \emph{not} imposed: it would force a sign on the normal derivative of every test function, which no subsolution need satisfy. It does hold for the value function at its points of differentiability, by Lemma~\ref{lem:normalsign}.
\end{definition}

The supersolution condition requires more care: the ``obvious'' dual of the subsolution is the wrong condition, so we motivate \eqref{eq:super-bellman} before stating it.

\emph{The naive dual is empty.} By analogy with the subsolution one might expect the reversed max-form $\max\!\bigl(|\taugrad\varphi|-1,-\nderiv\varphi\bigr)\ge0$. This is far too weak: it holds for \emph{every} test function with $\nderiv\varphi\le0$ --- in particular for constants, for which $\nabla\varphi=0$ --- so it carries no information at crack points and cannot yield a comparison principle. A usable supersolution condition must be genuinely sharper on $\C$.

\emph{What the condition must deliver.} Two things. It must be \emph{satisfied by the value function}, which is where the choice between the two half-spheres is decided, as Remark~\ref{rem:strongsuper} shows; and it must be strong enough to drive the comparison Lemma~\ref{lem:superopt}. The condition we adopt is metric: it requires $v$ to decrease at unit rate in some direction. The Bellman inequality below then follows from it (Proposition~\ref{prop:metric-implies-bellman}), so nothing of the differential formulation is lost, while the comparison argument gets the one statement it can use at a crack tip, where --- as we explain after the definition --- no $C^1$ test function need touch a supersolution from below at all. Remark~\ref{rem:weakclass} records what each of the two forms contributes.

\emph{The condition is the pointwise form of the dynamic programming principle.}
Let $\varphi\in C^1(\overline\Omega)$ touch $u$ from below at $x_0$. Inserting the lower bound of the dynamic programming principle (Lemma~\ref{lem:DPP}), namely $u(x_0)\ge u(x_0-hd_h)+h+o(h)$ for some $d_h\in A^{\mathrm{arr}}(x_0)$, using $u\ge\varphi$ with equality at $x_0$, expanding
$\varphi(x_0)-\varphi(x_0-hd_h)=hd_h\cdot\nabla\varphi(x_0)+o(h)$, and letting
$h\to0^+$, gives precisely $\sup_{d\in A^{\mathrm{arr}}(x_0)}
d\cdot\nabla\varphi(x_0)\ge1$. Thus \eqref{eq:super-bellman} is not an ad hoc inequality but the viscosity form of the DPP lower bound --- the supersolution condition for the control (Bellman) Hamiltonian
\[
H(x,p)=\sup_{d\in A^{\mathrm{arr}}(x)}d\cdot p-1,
\]
which is the crack analogue of the state-constraint boundary condition of Soner
\cite{Soner1986}.

\emph{Why the arrival set, and not the velocity set.} The dynamic programming principle is a \emph{backward} statement: it involves the directions from which $x_0$ can be \emph{reached}. If $x_0\in\C$ is reached along an admissible curve, its predecessor $x_0-hd$ must lie on the near side, $s_\C(x_0-hd)\le0$, that is $d\cdot n_\C\ge0$. These are the arrival directions $A^{\mathrm{arr}}(x_0)$ of~\eqref{eq:Aarr}, \emph{not} the velocity set $A(x_0)=\{|d|=1,d\cdot n_\C\le0\}$, which constrains \emph{forward} departure from the crack. Conflating the two reverses the operative half-sphere on $\C$ and produces a sign error there; in the bulk the distinction is invisible, since both sets are the full unit sphere and \eqref{eq:super-bellman} reduces to the classical $|\nabla\varphi|\ge1$. That the arrival choice is the correct one --- of the two half-spheres it is the one for which the value function satisfies \eqref{eq:super-bellman} at every crack point, the velocity-set supremum dropping below $1$ at the illuminated points --- is verified by the explicit computation in Remark~\ref{rem:strongsuper}.

\begin{definition}[Viscosity supersolution]\label{def:super}
A function $v\in\mathrm{LSC}(\widehat\Omega)$ is a \emph{viscosity supersolution} of~\eqref{eq:HJVI} if for every $x_0\in\widehat\Omega\setminus\{x^*\}$ with $v(x_0)<\infty$ there is $r_0=r_0(x_0)>0$ such that
\begin{equation}\label{eq:super-metric}
\inf\bigl\{v(\hat z):\dA(x_0,\hat z)=r\bigr\}\ \le\ v(x_0)-r
\qquad\text{for every } r\in(0,r_0) .
\end{equation}
\end{definition}

Three comments on this definition, which is metric rather than differential.

\emph{It is the dynamic programming principle.} Condition \eqref{eq:super-metric} says that $v$ decreases at unit rate in \emph{some} direction. The value function satisfies it with equality: if $\gamma$ is a minimizing geodesic from $x^*$ to $x_0$ and $r<u(x_0)$, the point of $\gamma$ at distance $r$ before $x_0$ lies on the sphere and has $u=u(x_0)-r$. The threshold $r_0$ must be allowed to depend on $x_0$ --- for the value function $r_0(x_0)=u(x_0)$, which degenerates as $x_0\to x^*$ --- and the condition is imposed only where $v$ is finite, since on a component of $\widehat\Omega$ not containing $x^*$ the value function is identically $+\infty$ and nothing is being asserted. The spheres appearing in \eqref{eq:super-metric} are non-empty for small $r$: each component of $\widehat\Omega$ is a geodesic space of positive diameter, a lip always having its non-empty near side in the same component by the nondegeneracy recorded in Section~\ref{sec:formulation}.

\emph{It implies the Bellman inequality.} The differential form of the condition is a consequence, not an extra requirement.

\begin{proposition}\label{prop:metric-implies-bellman}
Let $v$ be a viscosity supersolution in the sense of Definition~\ref{def:super}. Then for every $\varphi\in C^1(\overline\Omega)$ and every $x_0\in\widehat\Omega\setminus\{x^*\}$ with $v(x_0)<\infty$ at which $v-\varphi$ attains a local minimum,
\begin{equation}\label{eq:super-bellman}
\sup_{d\in A^{\mathrm{arr}}(x_0)} d\cdot\nabla\varphi(x_0)\ \ge\ 1 .
\end{equation}
\end{proposition}

\begin{proof}
Normalize so that $v\ge\varphi\circ\pi$ near $x_0$ with equality at $x_0$. Fix small $r>0$ and pick, by \eqref{eq:super-metric}, a point $\hat z_r$ with $\dA(x_0,\hat z_r)=r$ and $v(\hat z_r)\le v(x_0)-r+r^2$. Then
\[
\varphi(\pi\hat z_r)\ \le\ v(\hat z_r)\ \le\ \varphi(x_0)-r+r^2 .
\]
Write $\varrho_r:=|\pi\hat z_r-x_0|\le\dA(x_0,\hat z_r)=r$. We cannot have $\varrho_r=0$: that would force $\varphi(\pi\hat z_r)=\varphi(x_0)$ and hence $0\le-r+r^2$, false for small $r$. So $d_r:=(x_0-\pi\hat z_r)/\varrho_r$ is a unit vector, and a Taylor expansion of $\varphi$ gives
\[
\varrho_r\,d_r\cdot\nabla\varphi(x_0)\ \ge\ r-r^2+o(\varrho_r),
\qquad\text{whence}\qquad
d_r\cdot\nabla\varphi(x_0)\ \ge\ 1-r+o(1) ,
\]
using $\varrho_r\le r$. Moreover $\hat z_r$ lies in $\overline{K(x_0)}$ for $r$ small, so $\pi\hat z_r-x_0$ is within $o(\varrho_r)$ of the tangent cone $T(x_0)$, i.e. $\operatorname{dist}(d_r,A^{\mathrm{arr}}(x_0))\to0$. Passing to a subsequence $d_r\to d\in A^{\mathrm{arr}}(x_0)$, which is compact, gives $d\cdot\nabla\varphi(x_0)\ge1$.
\end{proof}

\emph{The converse is not available, which is why the metric form is taken as the definition.} At a crack tip the competitor that the comparison argument of Lemma~\ref{lem:superopt} needs to touch a supersolution from below admits no $C^1$ minorant, so the differential condition says nothing about it there. The obstruction is the shadow. Let $\C=[-1,0]\times\{0\}$ in the plane, $\hat x=(0,1)$, and let $p=(0,0)$ be the tip; then $\dA(\hat x,\cdot)$ equals $1+|\cdot-p|$ along the shadow lip behind $p$, while it equals $\sqrt{1+t^2}$ at distance $t$ from $p$ along the tangent line beyond the tip, where the crack has ended. So it increases \emph{conically} at $p$ and has no supergradient there. Consequently $\phi(\hat y):=c-\lambda\,\dA(\hat x,\hat y)^2$ admits no $C^1$ minorant touching it at $p$: writing $q$ for the gradient of a candidate, the direction $+e_1$ beyond the tip gives $\phi(p+te_1)-\phi(p)=-\lambda t^2$, hence $q\cdot e_1\le0$, while the direction $-e_1$ along the shadow lip gives $\phi(p-te_1)-\phi(p)=-\lambda(2t+t^2)$, hence $q\cdot e_1\ge2\lambda>0$. The two are incompatible. Separate conditions at the singular set are the rule rather than the exception for Hamilton--Jacobi equations on networks and stratified sets \cite{AchdouCamilliCutriTchou2013, BarlesBrianiChasseigne2013, ImbertMonneau2017}; here the metric form \eqref{eq:super-metric} discharges that role at every point at once, and metric formulations of this kind are standard for Hamilton--Jacobi equations on length spaces. The subsolution condition, by contrast, is unproblematic in test-function form and we keep it there.

\begin{remark}\label{rem:strongsuper}
We record the explicit form of \eqref{eq:super-bellman} promised above.

\emph{Bulk and tips.}
At a bulk point and at a tip, $A^{\mathrm{arr}}(x_0)$ is the full unit sphere, so $\sup_{d\in A^{\mathrm{arr}}} d\cdot\nabla\varphi = |\nabla\varphi(x_0)|$ and the condition is simply $|\nabla\varphi(x_0)|\ge1$. At a junction copy the supremum is over the unit vectors of $-T(x_0)$, and the condition is correspondingly stronger; nothing below depends on its explicit form there.

\emph{Regular crack points.}
At $x_0\in\C_{\mathrm{reg}}$, $A^{\mathrm{arr}}(x_0)=\{|d|=1,\,d\cdot\nC(x_0)\ge0\}$ is the hemisphere of directions pointing to the far side. Decompose $\nabla\varphi=\taugrad\varphi+(\nderiv\varphi)\,\nC$ into tangential and normal parts. Maximising $d\cdot\nabla\varphi$ over this hemisphere:
\begin{itemize}
\item If $\nderiv\varphi(x_0)\ge0$: the full gradient $\nabla\varphi$ itself points into the hemisphere (its normal component has the allowed sign), so the maximiser is $d=\nabla\varphi/|\nabla\varphi|$ and
$\sup = |\nabla\varphi(x_0)|$.
\item If $\nderiv\varphi(x_0)\le0$: the gradient points away from the hemisphere, so $d$ can only ``catch'' the tangential projection; the maximiser lies on the equator $\{d\cdot\nC=0\}$, giving $\sup = |\taugrad\varphi(x_0)|$.
\end{itemize}
In summary,
\begin{equation}\label{eq:sup-explicit}
\sup_{d\in A^{\mathrm{arr}}(x_0)} d\cdot\nabla\varphi(x_0)=
\begin{cases}
|\nabla\varphi(x_0)| & \text{if }\nderiv\varphi(x_0)\ge0,\\[4pt]
|\taugrad\varphi(x_0)| & \text{if }\nderiv\varphi(x_0)\le0.
\end{cases}
\end{equation}

\emph{Both cases carry the eikonal inequality.}
Since $A^{\mathrm{arr}}(x_0)$ consists of unit vectors, $\sup_{d\in A^{\mathrm{arr}}(x_0)}d\cdot\nabla\varphi(x_0)\le|\nabla\varphi(x_0)|$ always, so \eqref{eq:super-bellman} implies
\[
|\nabla\varphi(x_0)|\ge1
\qquad\text{at every point, crack points included:}
\]
directly in the first case of~\eqref{eq:sup-explicit}, and in the second because $|\nabla\varphi|\ge|\taugrad\varphi|\ge1$ there.
Conversely, \eqref{eq:super-bellman} is strictly \emph{stronger} than $|\nabla\varphi|\ge1$ on a lip: when $\nderiv\varphi(x_0)<0$ it demands $|\taugrad\varphi(x_0)|\ge1$, which the bulk eikonal inequality does not give.
This asymmetry is the state-constraint content of the lip condition, and no sign hypothesis on $\nderiv\varphi$ is needed to extract it.

\emph{Which case is operative for the value function.}
The first one: by Lemma~\ref{lem:normalsign}, at every lip point with $0<u<\infty$ at which $u$ is differentiable up to the lip one has $\nderiv u\ge0$ together with $|\nabla u|=1$, illuminated points included. Writing $\theta\in[0,\pi/2]$ for the angle between $\nabla u(x_0)$ and the tangent space $T_{x_0}\C$, so that $\nderiv u=\cos\theta$ and $|\taugrad u|=\sin\theta$, the two regimes of Section~\ref{sec:freeboundary} are the cases $\theta=\pi/2$ (sliding) and $\theta<\pi/2$ (illuminated).
We stress that this is a property of the value function, not a requirement placed on test functions: the crack being a two-sided barrier, the values of $u$ beyond a lip belong to the other lip and are unrelated to those on it, so no sign can be imposed on the normal derivative of a competitor at the level of Definition~\ref{def:super}.
Had the velocity set $A(x_0)=\{d\cdot\nC\le0\}$ been used in place of $A^{\mathrm{arr}}(x_0)$, the supremum at an illuminated point would equal $|\taugrad u|=\sin\theta<1$, and the value function would fail to be a supersolution; this is the sign error that the arrival set removes.
Step~2 of the proof of Theorem~\ref{thm:main} verifies that the value function satisfies \eqref{eq:super-bellman} at every point of $\widehat\Omega\setminus\{x^*\}$ at which it is finite. The source must be excluded: $\varphi\equiv0$ touches $u$ from below at $x^*$ with $\nabla\varphi(x^*)=0$, so \eqref{eq:super-bellman} fails there, as it must --- it is the failure of the supersolution condition at $x^*$ that pins the solution to the prescribed source, and it is what forces the descent polygon of Lemma~\ref{lem:superopt} to terminate there.
\end{remark}

\begin{definition}[Viscosity solution]
$u$ is a \emph{viscosity solution} of~\eqref{eq:HJVI} if it is a subsolution on $\widehat\Omega$, a supersolution on $\widehat\Omega\setminus\{x^*\}$, continuous on $\widehat\Omega$ in the sense of Remark~\ref{rem:slit}, and $u(x^*)=0$.
\end{definition}

\textbf{Remarks:} On each lip the crack condition encodes two competing mechanisms, which the value function realizes as follows:

\begin{itemize}
\item $\partial_{n_\C} u \ge 0$: the value function does not decrease toward the crack along the normal of the lip, which is the trace of non-penetration --- a trajectory gains nothing by pressing against the barrier.

\item $|\nabla_\tau u| \le 1$: propagation along the crack occurs at speed at most one, with equality exactly where sliding is optimal.
\end{itemize}

Both are properties of the value function and both follow from Lemma~\ref{lem:normalsign}: that lemma gives $\nderiv u\ge0$ directly, and $|\taugrad u|\le1$ because $|\taugrad u|^2=1-(\nderiv u)^2$. Neither is a condition imposed on test functions (Remark~\ref{rem:strongsuper}). Where they bind they do so simultaneously, by Lemma~\ref{lem:normalsign}, and combine into
\[
\max\big(|\nabla_\tau u| - 1,-\partial_{n_\C} u\big) = 0 ,
\]
which on the \emph{binding} (sliding) part of a lip asserts $\nderiv u=0$ and $|\taugrad u|=1$ at once --- equivalent conditions there, not alternatives. On the illuminated part, by contrast, both arguments of the max are strictly negative and the pointwise identity $\max(\cdot)=0$ does not hold; the crack condition there is the state constraint $\nderiv u\ge0$, $|\taugrad u|\le1$ without complementarity (see Remark~\ref{rem:reconcile}). This naturally leads to the emergence of a free boundary separating the two regimes.

\section{Main Results}

In this section, we establish the well\-posedness of the Hamilton--Jacobi variational inequality \eqref{eq:HJVI} and characterize the associated optimal trajectories.

\subsection{Existence and uniqueness}

We can now state the main well-posedness result.

\begin{theorem}\label{thm:main}
Assume that:
\begin{itemize}
\item $\Omega \subset \mathbb{R}^n$ is bounded with $C^1$ boundary,
\item $\C \subset \overline{\Omega}$ satisfies the standing hypotheses of Section~\ref{sec:formulation}: a finite union of compact $C^1$ sheets, meeting one another and $\partial\Omega$ transversally,
\item the source satisfies $x^*\in\Omega\setminus\C$.
\end{itemize}
No \emph{upper} semicontinuity of the set-valued maps \eqref{eq:A}--\eqref{eq:Aarr} is assumed, and none is available: both jump down at a lip, the full sphere at nearby bulk points collapsing to a hemisphere on the lip. (They are lower semicontinuous, which is not what a closed-graph argument would need.) The proofs below never need it --- this is one of the dividends of working in $\widehat\Omega$, where the descent of Lemma~\ref{lem:superopt} follows geodesics rather than a direction field.

Then the value function
\[
u(x)=\inf_{\gamma}\int_0^T |\dot{\gamma}(t)|dt,
\]
where the infimum is taken over all admissible trajectories in the sense of Definition~\ref{def:admissible}, is the unique viscosity solution of \eqref{eq:HJVI} in the class of functions that are continuous on $\widehat\Omega$, locally Lipschitz in $\overline\Omega\setminus\C$, and vanish at $x^*$. It is $1$-Lipschitz with respect to the admissible path metric $\dA$; in particular it is continuous on $\overline\Omega\setminus\C$ and continuous on the slit completion $\widehat\Omega$. If $\widehat\Omega$ is disconnected --- as happens when the crack closes around part of the domain --- the statement is to be read in the connected component of $x^*$, $u$ being identically $+\infty$ on the others by Remark~\ref{rem:lift}.
\end{theorem}

We first establish the Dynamic Programming Principle (DPP), which is the analytical engine driving both the sub- and supersolution verifications.

\begin{lemma}\label{lem:DPP}
Let $x\in\widehat\Omega$ with $0<u(x)<\infty$.
\begin{enumerate}
\item[\emph{(a)}] \emph{(Upper bound.)} For every $d\in A^{\mathrm{arr}}(x)$ there is a unit-speed admissible curve $\sigma_d:[0,h_0]\to\widehat\Omega$ with $\sigma_d(0)=x$ and
\begin{equation}\label{eq:DPPcurve}
\pi\sigma_d(h)=x-hd+O(h^2)\qquad(h\to0^+),
\end{equation}
such that
\begin{equation}\label{eq:DPPupper}
u(x)\ \le\ u\bigl(\sigma_d(h)\bigr)+h
\qquad (0<h\le h_0),
\end{equation}
with no error term. When $-d$ lies in the interior of $T(x)$ one may take $\sigma_d(h)=x-hd$ and \eqref{eq:DPPcurve} is an identity.
\item[\emph{(b)}] \emph{(Lower bound.)} There are $\rho_h\in(0,h]$ with $\rho_h\to0$ and unit vectors $d_h$ with $\operatorname{dist}\bigl(d_h,A^{\mathrm{arr}}(x)\bigr)\to0$ such that
\begin{equation}\label{eq:DPPlower}
u(x)\ \ge\ u(x-\rho_h d_h)+\rho_h .
\end{equation}
In particular every limit point of $(d_h)$ lies in the closed set $A^{\mathrm{arr}}(x)$.
\end{enumerate}
If $u(x)=+\infty$ both statements are vacuous, $u$ being then identically $+\infty$ on the component of $x$ in $\widehat\Omega$.
\end{lemma}

\begin{proof}
Throughout, trajectories are parametrized by arc length, so that $|\dot\gamma|=1$ a.e.\ and the cost of a trajectory equals its length. We work in $\widehat\Omega$ and write $x$ for $\hat x$, as agreed.

\medskip
\noindent\textbf{(a) Upper bound.}
Suppose first that $-d$ lies in the \emph{interior} of the tangent cone $T(x)$ of \eqref{eq:A}. Then for $h$ small the segment
\[
\eta:[0,h]\to\widehat\Omega,\qquad \eta(s):=x-hd+sd
\]
lies in $\overline{K(x)}$ and meets $\C$ only at its endpoint $x$; being a continuous curve in $\widehat\Omega$ it is admissible, by Remark~\ref{rem:lift}, and has length $h$. Concatenating an admissible curve from $x^*$ to $x-hd$ of cost at most $u(x-hd)+\delta$ with $\eta$ and letting $\delta\to0$ gives $u(x)\le u(x-hd)+h$, with no error term at all; so $\sigma_d(h)=x-hd$ serves.

For $-d$ on the \emph{boundary} of $T(x)$ the straight segment is unavailable, and it is worth being explicit about why, since the point is easy to miss. That $T(x)$ is the closure of its interior does not make a boundary direction realizable by a segment: approximating $-d$ by interior directions $-d^{(k)}$ gives $u(x)\le u(x-hd^{(k)})+h$ only for $h\le h_k$ with $h_k\to0$, so one cannot fix $h$ and let $k\to\infty$. Nor is the limit point admissible. Take $\C$ an arc of the unit circle with two tips, $x_0$ a point of its \emph{inner} lip, and $d=\tau$ a unit tangent, so $\tau\cdot\nC=0$ and $\tau\in A^{\mathrm{arr}}(x_0)$; then $x_0-h\tau$ lies strictly \emph{outside} the disk, that is, on the far side, where no lip determined by $K(x_0)$ exists at all. For $\Omega=B_{3.5}(0)$, the arc $\{|\theta|\le170^\circ\}$ and $x^*=(3,0)$ we computed $u(x_0)=6.56$ at the inner lip point $(1,0)$ while $u((1,-h))=\sqrt{4+h^2}\to2$, so the chord form of the inequality fails by $4.55$ and the discrepancy \emph{grows} as $h\to0$. The same happens at a point of $\partial\Omega\setminus\C$ with $\partial\Omega$ curved, where $x-hd$ leaves $\overline\Omega$ and $u(x-hd)$ is not even defined.

The remedy is to move along the obstacle instead of along its tangent. By transversality $\overline{K(x)}$ has at $x$ a tangent cone equal to the closure of its interior, and its boundary consists of pieces of $\C$ and of $\partial\Omega$, all of class $C^{1,1}$. Hence every $-d\in T(x)$ is realized by a $C^{1,1}$ curve in $\overline{K(x)}$ issuing from $x$ with initial velocity $-d$: for $-d$ interior take the segment; for $-d$ tangent to the boundary piece through $x$ take the geodesic of that piece --- on a lip, the arc of $\C$ along that lip, which lies in $\overline{K(x)}$ and is admissible along it by Remark~\ref{rem:lift}; for $-d$ on the boundary of a junction sector take the corresponding sheet. Reparametrized by arc length this is $\sigma_d$, and $C^{1,1}$ regularity gives the second-order expansion \eqref{eq:DPPcurve}, the $O(h^2)$ being controlled by the curvature bound of the piece. Concatenating an admissible curve from $x^*$ to $\sigma_d(h)$ with the reversal of $\sigma_d|_{[0,h]}$, which has length exactly $h$, gives \eqref{eq:DPPupper}. At a tip $x=p$ the set $\overline{K(p)}$ is a whole half-neighborhood of $p$ in $\overline\Omega$, $T(p)=\mathbb R^n$, and segments serve for every $d$; when $-d$ is tangent to the sheet the segment runs along $\C$ and is admissible along either lip, so \eqref{eq:DPPupper} then holds with $u(\sigma_d(h))$ read on either one.

\medskip
\noindent\textbf{(b) Lower bound.}
The slit completion $\widehat\Omega$ is a compact length space with intrinsic metric $\dA$ (Remark~\ref{rem:lift}), hence complete and locally compact; by the Hopf--Rinow--Cohn-Vossen theorem any two of its points at finite distance are joined by a minimizing geodesic. Since $u(x)=\dA(x^*,x)<\infty$, let $\gamma:[0,T]\to\widehat\Omega$ be a unit-speed minimizing geodesic from $x^*$ to $x$, $T=u(x)>0$.

For $0<h<T$ put $x_h:=\gamma(T-h)$, $\rho_h:=|x-x_h|$ and $d_h:=(x-x_h)/\rho_h$. We claim $\rho_h>0$. A minimizing geodesic is injective, so $x_h\ne x$ \emph{in $\widehat\Omega$}; but $\pi$ is not injective, so this alone does not give $\rho_h>0$. Suppose $\rho_h=0$, i.e.\ $\pi(x_h)=\pi(x)$ with $x_h\ne x$: the two then lie over the same point of $\C$ on different lips or in different junction components, and the sub-arc $\gamma|_{[T-h,T]}$ is a path in $\widehat\Omega$ joining them. Any such path must leave the tubular neighborhood of that crack point and round a tip or circumnavigate a junction, so its length is bounded below by a positive constant $c=c(x)$; this constant degenerates as $x$ approaches a tip, but $x$ is fixed here, and for $h<c(x)$ the situation is impossible. Hence $\rho_h>0$, and $\rho_h\le h\to0$. The restriction $\gamma|_{[0,T-h]}$ is admissible from $x^*$ to $x_h$ and has length $T-h$, so $u(x_h)\le u(x)-h$ and, as $x_h=x-\rho_hd_h$ and $\rho_h\le h$,
\[
u(x)\ \ge\ u(x-\rho_hd_h)+h\ \ge\ u(x-\rho_hd_h)+\rho_h ,
\]
which is \eqref{eq:DPPlower}. It remains to see that $d_h$ approaches $A^{\mathrm{arr}}(x)$. The point $x_h$ lies in $\overline{K(x)}$ for $h$ small, so $-\rho_hd_h=x_h-x$ is at distance $o(\rho_h)$ from the tangent cone $T(x)$, whence $\operatorname{dist}(-d_h,T(x))\to0$ and $\operatorname{dist}(d_h,A^{\mathrm{arr}}(x))\to0$. On a lip this is the familiar computation
\[
\nC(x)\cdot(x-x_h)=s_\C(x)-s_\C(x_h)+o(\rho_h)\ \ge\ o(\rho_h),
\qquad\text{i.e.}\qquad d_h\cdot\nC(x)\ge-o(1) .
\]
Finally, $A^{\mathrm{arr}}(x)$ being compact, every limit point of $(d_h)$ belongs to it.

\medskip
Statement \eqref{eq:DPPlower} is exact --- it carries no error term --- because the step is measured by the chord $\rho_h$ rather than by the arc length $h$. This is the only form of the dynamic programming principle used below, and it is what makes the verification of the sub- and supersolution properties free of any hypothesis on the regularity of optimal trajectories; compare the classical statements in Bardi--Capuzzo-Dolcetta \cite{BardiCapuzzoDolcetta1997}.
\end{proof}

The dynamic programming principle already settles the sign of the normal derivative of the value function, a fact used repeatedly below and, in the literature on one-sided crack models, usually postulated.

\begin{lemma}[Normal sign at a lip]\label{lem:normalsign}
Let $x_0\in\C_{\mathrm{reg}}$ with $0<u(x_0)<\infty$, and suppose $u$ is differentiable at $x_0$ up to one of the two lips, with $\nC$ the normal of that lip. Then
\[
|\nabla u(x_0)|=1
\qquad\text{and}\qquad
\nderiv u(x_0)\ \ge\ 0 .
\]
\end{lemma}

\begin{proof}
First, $|\nabla u(x_0)|\le1$. This needs the comparison in the direction opposite to the obvious one: $1$-Lipschitzness for $\dA$ bounds $u$ by $\dA$, and $\dA\ge|\cdot|$ would be useless, so what is required is $\dA(x_0,y)\le|y-x_0|\,(1+o(1))$ for $y$ in the closed near side $\overline{K(x_0)}$. That estimate holds because the sheet is a $C^1$ graph near $x_0$: two points of $\overline{K(x_0)}$ at Euclidean distance $\varrho$ are joined inside $\overline{K(x_0)}$ by a path of length $\varrho(1+o(1))$ as $\varrho\to0$. Combining, $|u(y)-u(x_0)|\le|y-x_0|(1+o(1))$ on the near side, whence $|\nabla u(x_0)|\le1$.
By the lower bound \eqref{eq:DPPlower} of Lemma~\ref{lem:DPP} there are $\rho_h\downarrow0$ and unit $d_h$ approaching $A^{\mathrm{arr}}(x_0)$ with $u(x_0)\ge u(x_0-\rho_hd_h)+\rho_h$. The predecessor $x_0-\rho_hd_h$ lies in $\overline{K(x_0)}$, so differentiability of $u$ up to the lip applies to it and gives
\[
\rho_h\,d_h\cdot\nabla u(x_0)+o(\rho_h)\ \ge\ \rho_h ,
\qquad\text{i.e.}\qquad d_h\cdot\nabla u(x_0)\ge1+o(1) .
\]
Passing to a subsequence $d_h\to d$, which lies in the compact set $A^{\mathrm{arr}}(x_0)$ by Lemma~\ref{lem:DPP}, we obtain
\[
d\cdot\nabla u(x_0)\ge1 .
\]
As $|d|=1$ this forces $|\nabla u(x_0)|\ge1$, hence $|\nabla u(x_0)|=1$ and equality in Cauchy--Schwarz, i.e. $d=\nabla u(x_0)$. Since $d\in A^{\mathrm{arr}}(x_0)$ satisfies $d\cdot\nC(x_0)\ge0$, we get $\nderiv u(x_0)=\nabla u(x_0)\cdot\nC(x_0)\ge0$.
\end{proof}

Geometrically, $\nabla u(x_0)$ \emph{is} the terminal direction of an optimal trajectory, and an optimal trajectory reaches the lip from the lip's own side; the inequality $\nderiv u\ge0$ records nothing more than that. Writing $\theta$ for the angle between that terminal direction and $\C$, Lemma~\ref{lem:normalsign} reads $\nderiv u(x_0)=\cos\theta\ge0$ and $|\taugrad u(x_0)|=\sin\theta\le1$, which is the trichotomy of Section~\ref{sec:freeboundary} in germ.

\medskip

\begin{proof}[Proof of Theorem~\ref{thm:main}]
The proof is organized into four steps: the subsolution property, the supersolution property, uniqueness, and regularity.

\bigskip
\noindent
\textbf{Step 1: Subsolution property.}

Let $\varphi\in C^1(\overline{\Omega})$ and suppose that
$u-\varphi$ attains a local maximum at $x_0\in\widehat\Omega\setminus\{x^*\}$, with $u(x_0)<\infty$; both provisos are part of Definition~\ref{def:sub} and are discussed there. Without loss of generality, we may assume that
\[
(u-\varphi)(x_0)=0
\qquad\text{and}\qquad u-\varphi\le0
\]
in a neighborhood of $x_0$. Fix $d\in A^{\mathrm{arr}}(x_0)$ and let $\sigma_d$ be the admissible curve supplied by the upper bound \eqref{eq:DPPupper} of the dynamic programming principle, so that for every small $h>0$
\[
u(x_0) \le u\bigl(\sigma_d(h)\bigr)+h .
\]
Since $u-\varphi$ has a local maximum at $x_0$ and $\sigma_d(h)\to x_0$ in $\widehat\Omega$, we have
\[
u\bigl(\sigma_d(h)\bigr)-u(x_0) \le \varphi\bigl(\pi\sigma_d(h)\bigr)-\varphi(x_0)
\]
for all sufficiently small $h>0$. Substituting yields
\[
0 \le \varphi\bigl(\pi\sigma_d(h)\bigr)-\varphi(x_0)+h .
\]
By \eqref{eq:DPPcurve}, $\pi\sigma_d(h)=x_0-hd+O(h^2)$, so the Taylor expansion of $\varphi\in C^1$ gives
\[
\varphi\bigl(\pi\sigma_d(h)\bigr) = \varphi(x_0)-h\,d\cdot\nabla\varphi(x_0)+o(h) ,
\]
the $O(h^2)$ displacement contributing only $o(h)$. Dividing by $h$ and letting $h\to0^+$,
\begin{equation}\label{eq:sub-fund}
d\cdot\nabla\varphi(x_0) \le 1
\qquad\text{for every } d\in A^{\mathrm{arr}}(x_0).
\end{equation}
It is worth noting what this replaces: with the straight chord $x_0-hd$ in place of $\pi\sigma_d(h)$, the local-maximum inequality would be applied at a point that need not lie in $\overline{K(x_0)}$ at all --- on a lip of negative curvature it lies on the far side, at $\dA$-distance bounded away from zero --- so the tangential bound~(i) below would be unproved exactly where it is used. Moving along the obstacle costs only the $O(h^2)$ above.

\emph{Case $x_0\in\Omega\setminus \C$.}
Then $A^{\mathrm{arr}}(x_0)=\{d\in\mathbb{R}^n:|d|=1\}$; choosing $d=\nabla\varphi(x_0)/|\nabla\varphi(x_0)|$ in \eqref{eq:sub-fund} (the case $\nabla\varphi(x_0)=0$ being trivial) gives
\[
|\nabla\varphi(x_0)| \le 1.
\]

\emph{Case $x_0\in\partial\C$ (a tip).}
The two sides merge at a tip, so $T(x_0)=\mathbb R^n$ and $A^{\mathrm{arr}}(x_0)=\{|d|=1\}$, exactly as in the bulk case; \eqref{eq:sub-fund} gives $|\nabla\varphi(x_0)|\le1$.

\emph{Case $x_0\in\partial\Omega\setminus\C$.}
Here $T(x_0)$ is the inward half-space and $A^{\mathrm{arr}}(x_0)$ the outward closed half-sphere, so \eqref{eq:sub-bellman} is the supremum of $d\cdot\nabla\varphi(x_0)$ over that half-sphere and \eqref{eq:sub-fund} is already the required inequality; no computation is needed. As at a junction, the bound is weaker than $|\nabla\varphi(x_0)|\le1$, and must be: fewer arrival directions are available at a boundary point.

\emph{Case $x_0\in\C_{\mathrm{self}}$ (a junction copy).}
Here $A^{\mathrm{arr}}(x_0)$ consists of the unit $d$ with $-d$ in the tangent cone $T(x_0)$ of the component occupied, in general a proper subset of the sphere, and \eqref{eq:sub-bellman} is the corresponding supremum. The resulting bound is \emph{weaker} than $|\nabla\varphi(x_0)|\le1$, as it must be: a trajectory confined to one sector of a polygonal crack has fewer arrival directions available, so less is being asserted about $\nabla\varphi$. No further computation is needed, \eqref{eq:sub-fund} being already the required inequality.

\emph{Case $x_0\in\C_{\mathrm{reg}}$ (a lip).}
Taking the supremum over $d\in A^{\mathrm{arr}}(x_0)$ in \eqref{eq:sub-fund} is exactly \eqref{eq:sub-bellman}, so the required inequality holds; we record what it says componentwise.

\emph{(i) Tangential bound.}
Every unit tangent vector $\tau\in T_{x_0}\C$ satisfies $\tau\cdot\nC(x_0)=0$, hence both $\tau$ and $-\tau$ belong to $A^{\mathrm{arr}}(x_0)$.
Applying \eqref{eq:sub-fund} with $d=\pm\tau$ gives $|\tau\cdot\nabla\varphi(x_0)|\le1$, and since $\tau\cdot\nabla\varphi(x_0)=\tau\cdot\taugrad\varphi(x_0)$, taking the supremum over all unit tangents yields
\[
|\taugrad\varphi(x_0)|\le1.
\]

\emph{(ii) Normal bound.}
Taking $d=\nC(x_0)\in A^{\mathrm{arr}}(x_0)$ in \eqref{eq:sub-fund} gives
\[
\nderiv\varphi(x_0)\le1 ,
\]
a bound which, with (i), does \emph{not} yield $|\nabla\varphi(x_0)|\le1$. The correct route uses a third choice of $d$: by \eqref{eq:sup-explicit} the supremum in \eqref{eq:sub-bellman} equals $|\nabla\varphi(x_0)|$ when $\nderiv\varphi(x_0)\ge0$ and $|\taugrad\varphi(x_0)|$ otherwise; in the first case $d=\nabla\varphi(x_0)/|\nabla\varphi(x_0)|$ lies in $A^{\mathrm{arr}}(x_0)$ and \eqref{eq:sub-fund} applied to it gives $|\nabla\varphi(x_0)|\le1$, while in the second case (i) already gives $|\taugrad\varphi(x_0)|\le1$. No \emph{lower} bound on $\nderiv\varphi(x_0)$ is available --- nor is one needed. The sign condition $\nderiv\varphi\ge0$ of the max-form cannot be imposed here: the crack being a two-sided barrier, a test function is free beyond the lip, where the values of $u$ belong to the other lip and bear no relation to those on this one, so a competitor may well have $\nderiv\varphi(x_0)<0$. Definition~\ref{def:sub} is formulated precisely so as not to require it, and \eqref{eq:sub-bellman} then reads $|\taugrad\varphi(x_0)|\le1$ in that case, which is what (i) supplies.

Therefore $u$ is a viscosity subsolution of \eqref{eq:HJVI}. For the value function itself the normal sign does hold, by Lemma~\ref{lem:normalsign}, so that
\[
\max\Big( |\nabla_\tau u(x_0)|-1,-\partial_{n_\C} u(x_0) \Big) \le 0
\]
holds for $u$ wherever it is differentiable up to the lip --- but as a theorem about $u$, not as part of the definition of a subsolution.

\bigskip
\noindent
\textbf{Step 2: Supersolution property.}

Let $\varphi \in C^1(\overline\Omega)$ and suppose that $u-\varphi$ attains a local minimum at $x_0$, with $(u-\varphi)(x_0)=0$ and $u\ge\varphi$ near $x_0$.

By the lower bound \eqref{eq:DPPlower} of Lemma~\ref{lem:DPP} there are $\rho_h\downarrow0$ and unit vectors $d_h$ approaching $A^{\mathrm{arr}}(x_0)$ with
\[
u(x_0)\ge u(x_0-\rho_hd_h)+\rho_h .
\]
Since $u\ge\varphi$ near $x_0$ and $u(x_0)=\varphi(x_0)$,
\[
\varphi(x_0)\ge \varphi(x_0-\rho_hd_h)+\rho_h ,
\]
and the Taylor expansion $\varphi(x_0)-\varphi(x_0-\rho_hd_h)=\rho_h\,d_h\cdot\nabla\varphi(x_0)+o(\rho_h)$ gives
\[
d_h\cdot\nabla\varphi(x_0)\ge1+o(1).
\]
Passing to a subsequence $d_h\to d\in A^{\mathrm{arr}}(x_0)$ --- legitimate because $A^{\mathrm{arr}}(x_0)$ is compact and $\operatorname{dist}(d_h,A^{\mathrm{arr}}(x_0))\to0$ --- we get $d\cdot\nabla\varphi(x_0)\ge1$ and hence
\[
\sup_{d\in A^{\mathrm{arr}}(x_0)} d\cdot\nabla\varphi(x_0)\ge1 .
\]
This is precisely \eqref{eq:super-bellman}, valid at every $x_0\in\widehat\Omega\setminus\{x^*\}$ at which $u$ is finite --- that is, at every such point of the component of $x^*$, which is where Lemma~\ref{lem:DPP} is available. (On the other components $u\equiv+\infty$, so $u-\varphi\equiv+\infty$ there and every point is trivially both a local maximum and a local minimum; the comparisons of Definitions~\ref{def:sub} and~\ref{def:super} are therefore understood, as usual, on $\{u<\infty\}$.) Therefore $u$ is a viscosity supersolution of \eqref{eq:HJVI}. It remains to record that $u$ satisfies Definition~\ref{def:super} itself, which is immediate and was noted there: for $x_0$ with $0<u(x_0)<\infty$ and $r<u(x_0)$, the point of a minimizing geodesic from $x^*$ to $x_0$ at distance $r$ before $x_0$ realizes \eqref{eq:super-metric} with equality. On a component not containing $x^*$ the value function is identically $+\infty$ and nothing is required of it, by the finiteness proviso in Definition~\ref{def:super}.

\bigskip
\noindent
\textbf{Step 3: Uniqueness.}


Uniqueness follows from the two comparison lemmas below (Lemmas~\ref{lem:subopt} and~\ref{lem:superopt}). Let $w$ be any viscosity solution in the class of the theorem. Lemma~\ref{lem:subopt} gives $w\le u$ on $\widehat\Omega$, and Lemma~\ref{lem:superopt} --- applicable since $w$, being continuous on $\widehat\Omega$, is in particular lower semicontinuous there --- gives $w\ge u$ on $\widehat\Omega$. Hence $w\equiv u$. The value function belongs to this class: it is continuous on $\widehat\Omega$ and locally Lipschitz in $\overline\Omega\setminus\C$ by Remark~\ref{rem:slit}, and $u(x^*)=0$.

\bigskip
\noindent
\textbf{Step 4: Regularity.}

By definition of the value function and concatenation of admissible curves,
\[
u(\hat y)\le u(\hat x)+\dA(\hat x,\hat y)
\qquad\text{for all } \hat x,\hat y\in\widehat\Omega,
\]
so $u$ is $1$-Lipschitz with respect to the admissible path metric on $\widehat\Omega$. Since for $\hat x,\hat y$ in the closure of one local component $\overline{K(\hat x)}$ one has $\dA(\hat x,\hat y)\le |\hat x-\hat y|(1+o(1))$ as $|\hat x-\hat y|\to0$, by the $C^1$ regularity of the sheets and of $\partial\Omega$, the function $u$ is locally Lipschitz --- hence continuous --- on $\overline\Omega\setminus\C$, and continuous on the slit completion $\widehat\Omega$ in the sense of Remark~\ref{rem:slit}.
This completes the proof of Theorem~\ref{thm:main}.
\end{proof}

\subsection{Comparison lemmas}\label{sec:comparison}


We first establish the suboptimality principle.

\begin{lemma}\label{lem:subopt}
Let $w$ be a viscosity subsolution of \eqref{eq:HJVI} with $w(x^*)\le0$, continuous on $\widehat\Omega$ and locally Lipschitz in $\overline\Omega\setminus\C$. Then $w\le u$ on $\widehat\Omega$.
\end{lemma}

\begin{proof}
We show that $w$ is $1$-Lipschitz with respect to $\dA$ on $\widehat\Omega$:
\begin{equation}\label{eq:nonexp}
w(\hat y)-w(\hat x)\ \le\ \dA(\hat x,\hat y)
\qquad\text{for all }\hat x,\hat y\in\widehat\Omega .
\end{equation}
Granting this, $w(x)\le w(x^*)+\dA(x^*,x)\le u(x)$ for every $x$, which is the assertion.

Since $\dA$ is the \emph{length} metric of $\widehat\Omega$ (Remark~\ref{rem:lift}), it suffices to work locally. Precisely, suppose that every $\hat x\in\widehat\Omega$ has a neighborhood $V$ in $\widehat\Omega$ such that
\begin{equation}\label{eq:loclip}
|w(\hat y)-w(\hat z)|\le d_V(\hat y,\hat z)
\qquad\text{for all }\hat y,\hat z\in V ,
\end{equation}
$d_V$ denoting the intrinsic distance \emph{of $V$}. Given $\hat x,\hat y$ and $\delta>0$, join them by an admissible curve of length at most $\dA(\hat x,\hat y)+\delta$, subdivide it by points $p_0,\dots,p_N$ so that each consecutive pair lies in one such $V$, and note that $d_V(p_i,p_{i+1})$ is at most the length of the enclosed sub-arc. Adding the estimates \eqref{eq:loclip} gives $|w(\hat x)-w(\hat y)|\le\dA(\hat x,\hat y)+\delta$, and $\delta$ was arbitrary. Using $d_V$ rather than $\dA$ in \eqref{eq:loclip} is what makes the subdivision legitimate: a near-minimizer of $\dA$ between two points of $V$ need not stay in $V$. The subsolution property is used only through its bulk consequence: by Definition~\ref{def:sub} the inequality $|\nabla w|\le1$ holds in the viscosity sense on the open set $\Omega\setminus\C$, and $w$ being locally Lipschitz there, $|\nabla w|\le1$ a.e.\ in $\Omega\setminus\C$.

\emph{Interior bulk points.} If $\hat x\in\Omega\setminus\C$, take for $V$ a ball around $\hat x$ contained in $\Omega\setminus\C$. Integrating $|\nabla w|\le1$ along segments gives \eqref{eq:loclip}, with $d_V$ the Euclidean distance.

\emph{All other points.} Let $\hat x$ lie over $x_0\in\overline\Omega$, and let $V$ be the closure in $\widehat\Omega$ of the connected component of $(B_r(x_0)\cap\Omega)\setminus\C$ determined by $\hat x$, for $r$ small enough that $B_r(x_0)$ lies inside a chart in which each sheet through $x_0$ is a graph. This single case covers lips over $\C_{\mathrm{reg}}$, tips, junction copies, and points of $\partial\Omega$, whether or not they lie on $\C$: in each of them $\hat x\in V$ and $V$ is a neighborhood of $\hat x$ in $\widehat\Omega$.

\emph{One-sided sub-case.} Suppose first that $V$ touches $\C\cup\partial\Omega$ from one side only, i.e.\ that there is a single vector $\nu$ pointing into the component from every point of $V\cap(\C\cup\partial\Omega)$ --- at a point of $\C_{\mathrm{reg}}$ the normal of the chart hyperplane, at a junction or a boundary point a vector interior to the tangent cone. Let $\hat y,\hat z\in V$ and let $c\subset V$ be a path joining them of length at most $d_V(\hat y,\hat z)+\delta$. Push $c$ into the interior of the component: by transversality and the $C^{1,1}$ regularity of the sheets the translates $c_\varepsilon:=c+\varepsilon\nu$ lie in $\Omega\setminus\C$ for small $\varepsilon$, have the same endpoints up to $O(\varepsilon)$, and have length $(1+o(1))\operatorname{length}(c)$ with $o(1)\to0$. Integrating $|\nabla w|\le1$ along $c_\varepsilon$, which lies in the open set $\Omega\setminus\C$ where $w$ is locally Lipschitz,
\[
|w(\hat y+\varepsilon\nu)-w(\hat z+\varepsilon\nu)|\le(1+o(1))\operatorname{length}(c) .
\]
Letting $\varepsilon\to0$ and using the continuity of $w$ on $\widehat\Omega$ --- which is exactly what makes the two lips separately accessible --- gives $|w(\hat y)-w(\hat z)|\le d_V(\hat y,\hat z)+\delta$, and $\delta>0$ was arbitrary.

\emph{Two-sided sub-case: a tip.} At a tip $p$ the set $V$ is a slit ball carrying \emph{both} lips of the sheet, and no single $\nu$ will do: for $\hat y$ on the upper lip one needs $\nu$ to point up and for $\hat z$ on the lower lip one needs it to point down, while $\nu=\pm e_1$ along the sheet would place the translated lip point on $\C$ itself. (We checked that no admissible direction exists: for $\C=[-0.5,0]\times\{0\}$, $p=0$, and $\hat y,\hat z$ the two lips over $(-0.1,0)$, a scan of $20\,000$ directions leaves none.) One therefore splits instead of translating. Any path in $V$ joining the two lips must cross the ray $\{x_1>0,\,x_2=0\}$ beyond the tip, at a point $q$ which is a \emph{bulk} point of $V$; let $c$ be a near-optimal path and $q$ such a crossing. Each of the two sub-arcs $\hat y\to q$ and $q\to\hat z$ lies in a one-sided sub-domain of $V$ --- the upper and lower halves of the slit ball, each of which touches $\C$ from one side --- so the previous sub-case applies to each with $\nu=\pm e_2$, and adding the two estimates gives $|w(\hat y)-w(\hat z)|\le\operatorname{length}(c)\le d_V(\hat y,\hat z)+\delta$. The same splitting handles a junction copy whose sector is bounded by two sheets meeting at $x_0$: a crossing point in the relative interior of the sector reduces matters to two one-sided pieces.

In both sub-cases only the component under consideration enters; the trace on an adjacent one is never compared with it, and no statement about the dimension of $\C_{\mathrm{sing}}$ is needed.

This establishes \eqref{eq:loclip} in every case, and with it \eqref{eq:nonexp}.

\emph{Two consequences recorded for later use.} The argument shows in particular that $w$ is $1$-Lipschitz for the geodesic distance of each connected open subset of $\Omega\setminus\C$ \emph{(bulk estimate)}, and that each one-sided trace $w_\pm$ is $1$-Lipschitz on each connected component of $\C_{\mathrm{reg}}$ for the intrinsic distance of that component \emph{(crack estimate)}, the latter because such an arc is itself a path in $\widehat\Omega$ along the corresponding lip. The componentwise form is the correct one and is enough for what follows; across a junction it genuinely fails, as the cross-shaped crack of Section~\ref{sec:viscosity} shows. Both are used in Section~\ref{sec:freeboundary}.
\end{proof}

We next turn to the superoptimality principle.

\begin{lemma}\label{lem:superopt}
Let $v\in \mathrm{LSC}(\widehat\Omega)$ be a viscosity supersolution of \eqref{eq:HJVI} on $\widehat\Omega\setminus\{x^*\}$ with $0\le v(x^*)<\infty$. Then $v\ge u$ on $\widehat\Omega$.
\end{lemma}

\begin{proof}
Fix $x_0\in\widehat\Omega$; we must show $v(x_0)\ge u(x_0)$. The strategy is to build, from the supersolution $v$, an admissible trajectory issuing from $x_0$ along which $v$ decreases at unit rate and which reaches $x^*$; a length comparison then gives the inequality. We may assume $v(x_0)<\infty$, there being nothing to prove otherwise; in particular $v\not\equiv+\infty$, so $\min_{\widehat\Omega}v$ is finite. The argument will show in passing that $x_0$ lies in the connected component of $x^*$ in $\widehat\Omega$, so that $u(x_0)<\infty$ as well.

Everything takes place in the connected component of $x_0$ in $\widehat\Omega$, which by Remark~\ref{rem:lift} is a compact length space with intrinsic metric $\dA$ (there are finitely many components, $\C$ being a finite union of compact sheets); by the Hopf--Rinow theorem for such spaces, any two points of one component are joined by a minimizing geodesic, and every sub-arc of such a geodesic is an admissible trajectory. This is what replaces, here, the Euclidean segments of the classical argument, and it is what makes the descent below automatically non-crossing.

\medskip
\noindent\emph{Step A: inf-convolution in the intrinsic metric.}
For $\varepsilon>0$ set
\[
v_\varepsilon(\hat x):=\min_{\hat y\in\widehat\Omega}
\Bigl\{v(\hat y)+\tfrac{1}{2\varepsilon}\dA(\hat x,\hat y)^2\Bigr\},
\qquad \hat x\in\widehat\Omega,
\]
the minimum being attained ($v$ being lower semicontinuous on that compact component) at some $\hat y_\varepsilon(\hat x)$. Taking $\hat y=\hat x$ as a competitor gives $v_\varepsilon\le v$, while $v_\varepsilon(\hat x)\ge\min_{\widehat\Omega}v+\tfrac{1}{2\varepsilon}\dA(\hat x,\hat y_\varepsilon(\hat x))^2$ by definition; hence
\begin{equation}\label{eq:infconv-est}
\dA\bigl(\hat x,\hat y_\varepsilon(\hat x)\bigr)^2\ \le\ 2\varepsilon\Bigl(v_\varepsilon(\hat x)-\min_{\widehat\Omega}v\Bigr) .
\end{equation}
The minimum of the lower semicontinuous $v$ on the compact $\widehat\Omega$ is finite; $v$ itself may take the value $+\infty$, which is why \eqref{eq:infconv-est} is written with $v_\varepsilon$ on the right. Along the descent of Step B the quantity $v_\varepsilon$ only decreases from $v_\varepsilon(x_0)\le v(x_0)<\infty$, so throughout that construction
\[
\dA\bigl(\xi_k,\hat y_\varepsilon(\xi_k)\bigr)\ \le\ r_\varepsilon:=\sqrt{2\varepsilon\bigl(v(x_0)-\textstyle\min_{\widehat\Omega}v\bigr)}\ \xrightarrow[\varepsilon\to0]{}\ 0 .
\] The intrinsic metric is used here in place of the Euclidean one for a single but decisive reason, made explicit in Step B: the minimizer is then reached from $\hat x$ along an \emph{admissible} geodesic, so the descent direction it defines is an arrival direction at $\hat x$ itself.

\emph{The slope bound.} We claim that for every $\hat x\ne x^*$,
\begin{equation}\label{eq:slope}
\rho_\varepsilon(\hat x):=\dA\bigl(\hat x,\hat y_\varepsilon(\hat x)\bigr)\ \ge\ \varepsilon
\qquad\text{whenever}\qquad \hat y_\varepsilon(\hat x)\ne x^* .
\end{equation}
Indeed, write $\hat y_*:=\hat y_\varepsilon(\hat x)$ and $\rho:=\rho_\varepsilon(\hat x)$, so that
\[
v_\varepsilon(\hat x)=v(\hat y_*)+\frac{\rho^2}{2\varepsilon} ,
\]
and suppose $\hat y_*\ne x^*$. Let $r>0$ be small and $\eta>0$. By Definition~\ref{def:super} applied at $\hat y_*$, where $v(\hat y_*)\le v_\varepsilon(\hat x)\le v(\hat x)<\infty$, there is for each small $r$ a point $\hat z$ with $\dA(\hat y_*,\hat z)=r$ and
\[
v(\hat z)\ \le\ v(\hat y_*)-r+\eta .
\]
Using $\hat z$ as a competitor in the minimum defining $v_\varepsilon(\hat x)$, and the triangle inequality $\dA(\hat x,\hat z)\le\rho+r$,
\[
v(\hat y_*)+\frac{\rho^2}{2\varepsilon}
= v_\varepsilon(\hat x)
\ \le\ v(\hat z)+\frac{\dA(\hat x,\hat z)^2}{2\varepsilon}
\ \le\ v(\hat y_*)-r+\eta+\frac{(\rho+r)^2}{2\varepsilon} .
\]
Cancelling and rearranging,
\[
r\ \le\ \eta+\frac{2\rho r+r^2}{2\varepsilon},
\qquad\text{so}\qquad
1\ \le\ \frac{\eta}{r}+\frac{2\rho+r}{2\varepsilon} .
\]
Letting $\eta\to0$ and then $r\to0$ gives $\rho\ge\varepsilon$, which is \eqref{eq:slope}.

The computation covers the degenerate case $\hat y_*=\hat x$ as well, and excludes it when $\hat x\ne x^*$: there $\rho=0$, contradicting $\rho\ge\varepsilon$.

Three features of this argument deserve emphasis. It uses no test function, no differentiability and no regularity of $\dA$, so it is insensitive to the failure of semiconcavity of $\dA(\hat x,\cdot)$ at the tips --- which is real, and is exactly why Definition~\ref{def:super} is metric. It uses the supersolution property only at the single point $\hat y_*$. And the constant in \eqref{eq:slope} is sharp: for $v=\dA(x^*,\cdot)$ in the absence of a crack one computes $\rho=\varepsilon$ exactly.

\medskip
\noindent\emph{Step B: an admissible descent polygon.}
Fix $h\in(0,\varepsilon)$ and define $\xi_0:=x_0$ and, inductively, $\xi_{k+1}:=\zeta_k(h)$, where $\zeta_k$ is a minimizing geodesic of $\widehat\Omega$ from $\xi_k$ to $\hat y_\varepsilon(\xi_k)$, parametrized by arc length. The scheme is stopped as soon as $\hat y_\varepsilon(\xi_k)=x^*$ or $\xi_k=x^*$.

\emph{Admissibility.} Each arc $\zeta_k|_{[0,h]}$ is a sub-arc of a minimizing geodesic of $\widehat\Omega$, hence a continuous curve in $\widehat\Omega$, hence admissible by Remark~\ref{rem:lift}; the concatenation $\sigma_h$ of the arcs is therefore an admissible trajectory of length $Nh$ after $N$ steps. No viability or inward-pointing argument is needed: the descent follows geodesics of the slit completion, which cannot cross $\C$ by construction. Equivalently, in the language of \eqref{eq:A}--\eqref{eq:Aarr}, the velocity $\dot\zeta_k(0)$ is the negative of the arrival direction at $\xi_k$ of the reversed geodesic, so it lies in $A(\xi_k)$ whenever that derivative exists.

\emph{Decrease.} Using $\hat y_*:=\hat y_\varepsilon(\xi_k)$ as a competitor in the minimum defining $v_\varepsilon(\xi_{k+1})$, and $\dA(\xi_{k+1},\hat y_*)=\rho_\varepsilon(\xi_k)-h$ because $\xi_{k+1}$ lies on a minimizing geodesic from $\xi_k$ to $\hat y_*$ at arc length $h$,
\[
\begin{aligned}
v_\varepsilon(\xi_{k+1})
&\le v(\hat y_*)+\frac{(\rho_\varepsilon(\xi_k)-h)^2}{2\varepsilon}
= v_\varepsilon(\xi_k)-\frac{h\bigl(2\rho_\varepsilon(\xi_k)-h\bigr)}{2\varepsilon}\\
&\le v_\varepsilon(\xi_k)-h+\frac{h^2}{2\varepsilon},
\end{aligned}
\]
the last step by the slope bound \eqref{eq:slope}. Note that this estimate is purely metric: it uses neither differentiability nor semiconcavity of $v_\varepsilon$, and in particular no generalized characteristics.

\emph{Termination.} Each step lowers $v_\varepsilon$ by at least $h(1-h/2\varepsilon)>0$, while $v_\varepsilon\ge\min_{\widehat\Omega}v>-\infty$; hence the scheme stops after finitely many steps, say $N$, and it can stop only through $\hat y_\varepsilon(\xi_N)=x^*$ or $\xi_N=x^*$. In either case
\[
\dA(\xi_N,x^*)\le r_\varepsilon .
\]

\medskip
\noindent\emph{Step C: length comparison.}
Summing the decrease over the $N$ steps gives
\[
N h\Bigl(1-\frac{h}{2\varepsilon}\Bigr)\ \le\ v_\varepsilon(x_0)-v_\varepsilon(\xi_N)\ \le\ v(x_0)-m_\varepsilon ,
\qquad
m_\varepsilon:=\min\bigl\{v(x^*),\,v_\varepsilon(x^*)\bigr\} ,
\]
where $v_\varepsilon\le v$ was used at $x_0$, and $v_\varepsilon(\xi_N)\ge m_\varepsilon$ holds in both stopping cases: if $\hat y_\varepsilon(\xi_N)=x^*$ then $v_\varepsilon(\xi_N)\ge v(\hat y_\varepsilon(\xi_N))=v(x^*)$, while if $\xi_N=x^*$ then $v_\varepsilon(\xi_N)=v_\varepsilon(x^*)$. Since inf-convolutions of a lower semicontinuous function increase to it pointwise, $m_\varepsilon\to v(x^*)\ge0$ as $\varepsilon\to0$.
The concatenation $\sigma_h$ is admissible from $x_0$ to $\xi_N$ and has length $Nh$, so
\[
\dA(x_0,x^*)\ \le\ Nh+\dA(\xi_N,x^*)\ \le\ \frac{v(x_0)-m_\varepsilon}{1-h/2\varepsilon}+r_\varepsilon .
\]
In particular $\dA(x_0,x^*)<\infty$, which is the finiteness announced at the outset: the descent polygon joins $x_0$ to $x^*$ inside $\widehat\Omega$. Since $\dA$ is symmetric (Definition~\ref{def:pathmetric}),
\[
u(x_0)=\dA(x^*,x_0)=\dA(x_0,x^*)\le\frac{v(x_0)-m_\varepsilon}{1-h/2\varepsilon}+r_\varepsilon .
\]
Letting $h\to0$ at fixed $\varepsilon$, and then $\varepsilon\to0$ (so that $r_\varepsilon\to0$ and $m_\varepsilon\to v(x^*)\ge0$), we obtain $u(x_0)\le v(x_0)$. As $x_0$ was arbitrary, $v\ge u$ on $\widehat\Omega$.
\end{proof}

\begin{remark}[What the comparison actually uses]\label{rem:weakclass}
The proof of Lemma~\ref{lem:superopt} uses the supersolution property only through the slope bound \eqref{eq:slope}, and that bound was obtained from \eqref{eq:super-metric} alone, at the single point $\hat y_\varepsilon(\hat x)$, with no test function and no differentiability. Together with Proposition~\ref{prop:metric-implies-bellman} this gives a clean division of labour: the metric condition \eqref{eq:super-metric} is what the comparison consumes, and the Bellman inequality \eqref{eq:super-bellman} --- which it implies --- is what identifies the problem as a Hamilton--Jacobi equation and distinguishes the arrival set from the velocity set.

We do not claim an enlargement of the uniqueness class. The converse implication, from \eqref{eq:super-bellman} back to \eqref{eq:super-metric}, is exactly a local superoptimality principle and is not established here; and at a crack tip the natural route to it is blocked, the competitor it would use admitting no $C^1$ minorant, as shown after Definition~\ref{def:super}. What is established is one implication, and the definition is taken on the side from which it points.
\end{remark}

Combining the two optimality principles yields a comparison principle, in which the value function is sandwiched between every subsolution and every supersolution.

\begin{theorem}\label{thm:comparison}
Let $w$ be a viscosity subsolution of \eqref{eq:HJVI} in the class of Lemma~\ref{lem:subopt}, and let $v\in\mathrm{LSC}(\widehat\Omega)$ be a viscosity supersolution on $\widehat\Omega\setminus\{x^*\}$ with $v(x^*)<\infty$.
If $w(x^*)\le0\le v(x^*)$, then
\[
w\le u\le v\quad\text{on }\widehat\Omega ;
\]
in particular $w\le v$, and $u$ is both the largest subsolution and the smallest supersolution with these boundary values at $x^*$.
\end{theorem}

\begin{proof}
Lemma~\ref{lem:subopt} gives $w\le u$ and Lemma~\ref{lem:superopt} gives $u\le v$.
\end{proof}

\begin{remark}[Neither lip dominates]\label{rem:lips}
One might hope to simplify the formulation by declaring one side of each sheet admissible and the other forbidden --- that is, by co-orienting $\C$ so that the cheaper lip is always the same one. For a convex sheet seen from outside, such as the circular arc of Section~\ref{sec:freeboundary}, this is indeed the case, and the source-facing lip is cheaper at every point. In general, however, no such choice exists, and this is why the two traces of Remark~\ref{rem:slit} must be carried independently throughout.

The obstruction is topological rather than metric, and a single smooth sheet already exhibits it. Let $n=2$ and let $\C$ be the Archimedean spiral
\[
\C=\bigl\{\bigl((1+\tfrac{\theta}{4})\cos\theta,\ (1+\tfrac{\theta}{4})\sin\theta\bigr):0\le\theta\le4\pi\bigr\},
\]
a single $C^1$ sheet with two free ends, and place the source $x^*=(-2.55,0)$ in the channel between the first and the second turn. Both lips of this sheet are reachable, since the sheet does not close. On the inner turn the channel containing $x^*$ lies \emph{outside} the sheet, so the outer lip is the cheap one there and the inner lip can be reached only by travelling to the inner free end; on the next turn the same channel lies \emph{inside} the sheet, and the roles are exchanged. The cheaper lip therefore switches along the sheet, at $\theta=1.886\pi$ --- arc length $10.43$ from the inner end --- where the two traces cross transversally; see Figure~\ref{fig:lips}. The computation is the one of Section~\ref{sec:numerics}, with both lips carrying graph nodes. The switch is not a near-tie produced by the discretization: the two traces differ by up to $16.5$ elsewhere on the sheet, and the crossing moves by less than $10^{-3}$ in $\theta$ when the lip offset is refined from $10^{-2}$ to $10^{-3}$.

Any co-orientation of this sheet would consequently declare the cheaper lip forbidden along one of the two portions, and the value function of the resulting one-sided problem would not be the geodesic distance of the slit domain. What \emph{is} true is that the coincidence set $\{u_+=u_-\}$ --- the locus reached at equal cost from both sides --- is closed and separates each sheet into the portions on which one or the other lip is the cheaper; on the spiral it is the single point $\theta=1.886\pi$. It plays no role in the analysis below, which is local to a lip; we record it because it is precisely the obstruction to a one-sided reformulation.
\end{remark}

\begin{figure}[htbp]
\centering
\includegraphics[width=\linewidth]{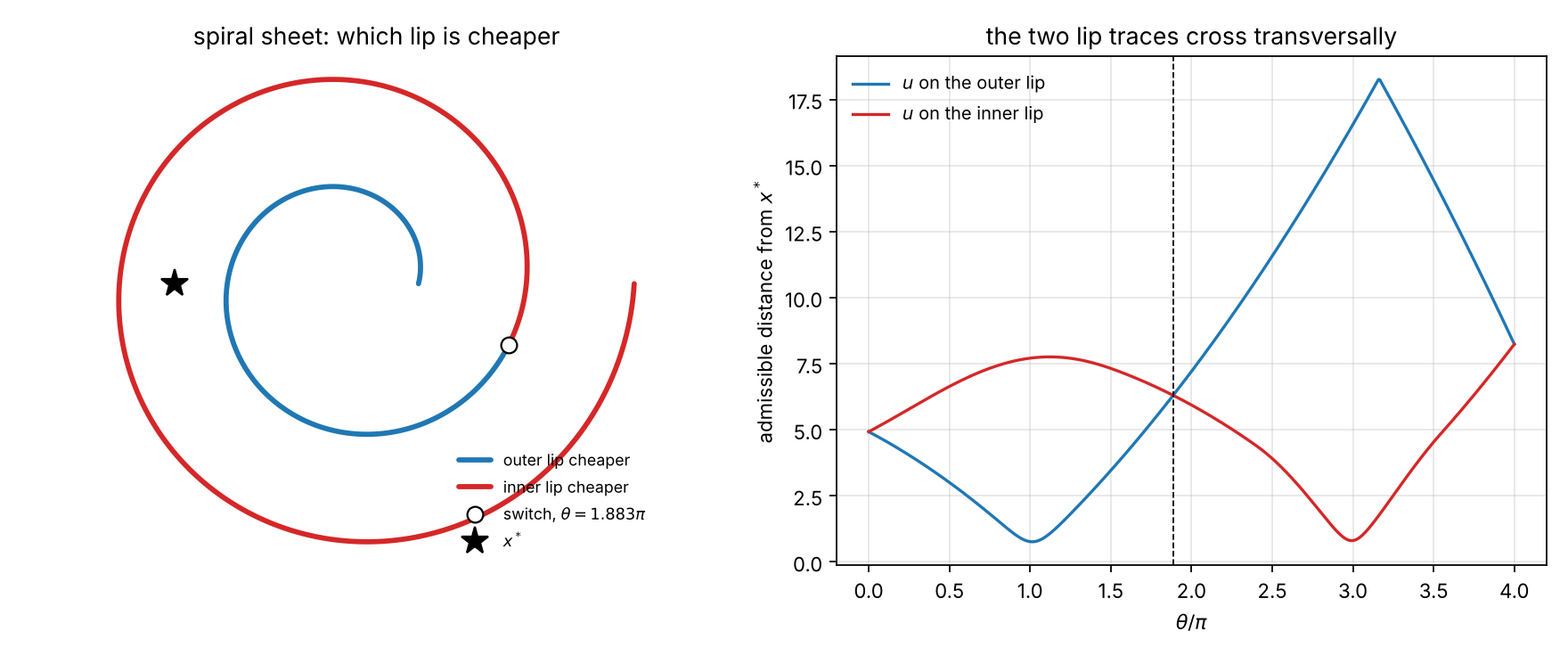}
\caption{Neither lip of a crack is systematically the cheaper one. Left: the spiral sheet of Remark~\ref{rem:lips}, coloured by which of its two lips is reached at lower cost from $x^*$ (star); the two colours meet at a single switching point (white circle). Right: the two lip traces of $u$ along the sheet, crossing transversally at $\theta=1.886\pi$. Computed on the visibility graph of the slit completion, with each crack sample duplicated into two lip nodes.}
\label{fig:lips}
\end{figure}

\begin{remark}[Why doubling of variables fails at the interface]
The classical doubling argument of Crandall, Ishii, and Lions~\cite{CrandallIshiiLions1992}, applied to \eqref{eq:HJVI}, breaks down when the minimizing point of the doubled functional lands on $\C$. The obstruction is present on both sides.

On the subsolution side, \eqref{eq:sub-bellman} bounds only $|\taugrad\varphi|$ when $\nderiv\varphi\le0$ and leaves the normal component of the doubling variable $p_\varepsilon$ unconstrained; hence $|p_\varepsilon|\le1$ need not hold at a crack point, and the usual contradiction $1\le|p_\varepsilon|\le1$ is unavailable.

On the supersolution side the obstruction is of another kind: the doubled functional is minimized over $\widehat\Omega\times\widehat\Omega$, whose two points may sit on opposite lips over the same crack point, where the penalization $|x-y|^2/\varepsilon$ does not force them together in the intrinsic metric. The two inequalities are then evaluated against opposite normals and cannot be combined.

The trajectory-based Lemmas~\ref{lem:subopt} and~\ref{lem:superopt} circumvent this by working along admissible curves rather than at a single doubled point; the superoptimality direction relies on the metric form \eqref{eq:super-metric} of the supersolution condition, the admissible descent direction being furnished by a minimizing geodesic of $\widehat\Omega$ rather than by any supremum.
\end{remark}

\subsection{Geometric characterization of optimal trajectories}


We now derive the tangential contact condition at an interior crack contact. It holds for every admissible trajectory, optimal or not.

\begin{theorem}\label{thm:tangential}
Let $\gamma$ be an \emph{admissible} trajectory and let $t_0$ be an interior time at which $\gamma$ touches the crack at a regular point $y=\gamma(t_0)\in\C_{\mathrm{reg}}$.
If $\dot\gamma(t_0)$ exists, then
\[
\dot{\gamma}(t_0)\cdot n_\C(y)=0 .
\]
Moreover, if $\gamma$ propagates along the regular part of the crack over a nontrivial interval,
\[
\gamma([t_1,t_2])\subset \C_{\mathrm{reg}},
\]
then
\[
\dot{\gamma}(t)\in T_{\gamma(t)}\C
\quad \text{for a.e. } t\in [t_1,t_2].
\]
\end{theorem}

\begin{proof}
Recall the local signed distance $s_\C$ on a tubular neighborhood $U$ of the regular point $y$ (positive on the far side), so that $s_\C\in C^1(U)$, $|\nabla s_\C|\equiv1$, and $\nabla s_\C=\nC$ on $\C\cap U$.

By admissibility (condition (ii) of Definition~\ref{def:admissible}), the trajectory does not cross to the far side, hence
\[
s_\C(\gamma(t))\le0
\quad\text{for all } t \text{ near } t_0,
\qquad
s_\C(\gamma(t_0)) = 0 .
\]
Thus $t_0$ is an interior maximum of the function $t\mapsto s_\C(\gamma(t))$. If $\dot\gamma(t_0)$ exists, this function is differentiable at $t_0$ with derivative
\[
\frac{d}{dt}\Big|_{t=t_0} s_\C(\gamma(t)) = \nabla s_\C(y)\cdot\dot\gamma(t_0) = n_\C(y)\cdot\dot\gamma(t_0),
\]
and the first-order condition at an interior maximum gives $n_\C(y)\cdot\dot\gamma(t_0)=0$, as claimed. Note that this is a genuine two-sided first-order condition, and that it comes from the non-crossing clause~(ii) of Definition~\ref{def:admissible}, not from optimality: the velocity constraint of clause~(i) alone would give only one inequality. Optimality plays no role in this theorem, which is why it may be applied freely below to sub-arcs and to competitors.

For the sliding regime, suppose $\gamma([t_1,t_2])\subset\C_{\mathrm{reg}}$. The arc being compact and contained in the regular part, it is covered by finitely many tubular neighborhoods of the kind used above, on each of which a signed distance $s_\C$ is defined; on each such piece $s_\C(\gamma(t))\equiv0$, and differentiating at a.e.\ $t$ gives $n_\C(\gamma(t))\cdot\dot\gamma(t)=0$. Since $\C_{\mathrm{reg}}$ is a $C^1$ hypersurface and the curve lies in $\C$, the vanishing of the normal component of the velocity is equivalent to
\[
\dot{\gamma}(t)\in T_{\gamma(t)}\C
\]
for almost every $t\in[t_1,t_2]$.
\end{proof}

As a consequence, an optimal trajectory never reflects off the crack at a regular point.

\begin{corollary}\label{cor:nocorner}
Let $n=2$, let $y\in\C_{\mathrm{reg}}$ be an isolated contact point of an optimal trajectory, and suppose that near $y$ the trajectory consists of two straight segments $[a,y]$ and $[y,b]$ contained in $\Omega$ and meeting $\C$ only at $y$. Suppose in addition that the closed triangle with vertices $a,y,b$ meets $\C$ only at $y$.
Then both segments are tangent to $\C$ at $y$ and
\[
\frac{y-a}{|y-a|}=\frac{b-y}{|b-y|},
\]
so the trajectory passes straight through $y$, grazing the crack.
In particular the equal-angle (Snell-type) reflection condition holds only in this degenerate form: a genuine reflection at a regular crack point is never optimal, and in the plane an optimal trajectory changes direction on $\C$ only at tips, at junctions, or along sliding arcs (Theorem~\ref{thm:structure}).
\end{corollary}

\begin{proof}
Write $e_{\mathrm{in}}:=(y-a)/|y-a|$ for the arrival direction and $e_{\mathrm{out}}:=(b-y)/|b-y|$ for the departure direction, and suppose $e_{\mathrm{in}}\ne e_{\mathrm{out}}$.
The closed triangle with vertices $a,y,b$ meets $\C$ only at $y$ by hypothesis, and the chord $[a,b]$ lies in that triangle and avoids $y$ (the three points not being collinear, since $e_{\mathrm{in}}\ne e_{\mathrm{out}}$). Hence $[a,b]$ meets $\C$ nowhere, lies in $\Omega$, and is admissible; and
\[
|b-a|<|y-a|+|b-y|
\]
strictly, by the triangle inequality. Replacing the sub-arc $a\to y\to b$ of $\gamma$ by $[a,b]$ therefore produces a strictly shorter admissible trajectory with the same endpoints, contradicting the optimality of $\gamma$.
Hence $e_{\mathrm{in}}=e_{\mathrm{out}}$, so $\dot\gamma(t_0)$ exists at the contact time, and Theorem~\ref{thm:tangential} gives $\dot\gamma(t_0)\cdot\nC(y)=0$: the common direction is tangent to $\C$.

The triangle hypothesis is in fact automatic, and it is worth recording why, since the two cases are governed by opposite curvature signs.

If the contact is \emph{transversal}, $e_{\mathrm{in}}\cdot\nC(y)>0$, the two segments leave $\C$ at a positive angle, so the triangle lies on the near side near $y$ and meets $\C$ only at $y$ once $|y-a|$ and $|b-y|$ are small.

If the contact is \emph{tangential}, the segments are tangent to $\C$ at $y$, and admissibility already forces the crack to lie locally on the \emph{far} side of that tangent line: were it on the near side, the segments --- which lie on the near side --- would cross it, contrary to the hypothesis that they meet $\C$ only at $y$. In the plane, writing $\kappa$ for the curvature in the convention of Theorem~\ref{thm:structure}, the crack is displaced from its tangent line by $\tfrac12 s^2\kappa\,\nC(y)+o(s^2)$ in the arc length $s$, so this says $\kappa(y)\ge0$, with $\kappa(y)>0$ giving the displacement at second order. The triangle, contained in the near side of the tangent line up to the same order, therefore again meets $\C$ only at $y$.

Thus a tangential contact at a point with $\kappa<0$ simply cannot occur, and the configuration one might have feared --- a corner that no chord competitor can remove --- is empty.
\end{proof}

\section{Free-boundary structure on the crack}\label{sec:freeboundary}

We now analyse how optimal trajectories interact with $\C$. Throughout this section $u$ denotes the value function of Theorem~\ref{thm:main}, and we work at regular points of $\C$, using the one-sided traces of $u$ on the slit completion $\widehat\Omega$ (Remark~\ref{rem:slit}). Everything below is local to one lip: we fix a lip, write $\nC$ for its outward normal and $u|_\C$ for the trace of $u$ on it --- which is $1$-Lipschitz for the intrinsic distance of each connected component of $\C_{\mathrm{reg}}$, by the crack estimate recorded at the end of the proof of Lemma~\ref{lem:subopt} --- and the whole trichotomy is read on that lip. The other lip of the same sheet carries its own copy of the analysis, with its own sliding and illuminated sets, and the two are unrelated (Remark~\ref{rem:lips}).

\subsection{The sliding set, the illuminated set, and the free boundary}

\begin{definition}[Crack regimes]\label{def:regimes}
Write $\widehat\C$ for the set of points of $\pi^{-1}(\C_{\mathrm{reg}})$ at which $u$ is differentiable up to the lip, a subset of $\widehat\Omega$ on which $\taugrad u$ and $\nderiv u$ are both defined (Lemma~\ref{lem:normalsign}); a regular crack point contributes at most two elements of $\widehat\C$, and may well contribute only one: differentiability up to one lip says nothing about the other. The sets below are subsets of $\widehat\C$, so each lip carries its own and the two are unrelated. Define
\begin{align*}
\Sigma & :=\{\hat y\in\widehat\C:|\taugrad u(\hat y)|=1\} \quad(\text{\emph{sliding set}}), \\
I      & :=\{\hat y\in\widehat\C:|\taugrad u(\hat y)|<1\} \quad(\text{\emph{illuminated set}}),
\end{align*}
and the \emph{free boundary}
\[
\Gamma:=\partial_{\widehat\C}\Sigma
\]
as the boundary of $\Sigma$ relative to $\widehat\C$.
\end{definition}

Since $u|_\C$ is $1$-Lipschitz intrinsically, $|\taugrad u|\le1$ holds wherever defined, so $\{\Sigma,I\}$ \emph{partitions} $\widehat\C$. The free boundary $\Gamma$ is not a third piece but the interface along which $I$ abuts $\Sigma$. We do \emph{not} claim $\Gamma\subset\Sigma$: that would require $\taugrad u$ to be continuous on $\widehat\C$, which is not part of the definition --- for a merely differentiable $1$-Lipschitz trace the set $\{|\taugrad u|=1\}$ need not be relatively closed. The statements below that are made at free-boundary points are therefore made at points of $\Gamma\cap\Sigma$, which is where they are used; in the worked example of Section~\ref{sec:freeboundary} the trace is real-analytic and $\Gamma\subset\Sigma$ holds outright. Points of $\C_{\mathrm{reg}}$ at which $u$ fails to be differentiable up to the lip lie in none of the three, and neither does $\C_{\mathrm{sing}}$. The names reflect the geometry made precise below: on $I$ the crack point is reached by a bulk ray arriving transversally (``illuminated'' by the source), while on $\Sigma$ the optimal trajectory runs \emph{along} the crack (``sliding'').

We now characterize the sliding regime.

\begin{proposition}\label{prop:sliding}
On $\Sigma$ the trace satisfies the intrinsic eikonal equation $|\nabla_{\!\C}(u|_\C)|=1$ pointwise, by the definition of $\Sigma$. Moreover, for $y\in\Sigma$ let $\gamma$ be an optimal trajectory reaching $y$ whose contact set $\{t:\gamma(t)\in\C\}$ is a finite union of points and intervals --- the tameness hypothesis already used in Theorem~\ref{thm:structure}(c), and not derived here. Then $\gamma$ falls into exactly one of two cases:
\begin{enumerate}
\item[\emph{(S)}] its terminal arc lies in $\C_{\mathrm{reg}}$ over a nondegenerate interval, and is then a unit-speed intrinsic geodesic of $\C$ \emph{(sliding)};
\item[\emph{(G)}] its terminal arc lies in the bulk and arrives \emph{tangentially}, $\nderiv u(y)=0$ \emph{(grazing)}.
\end{enumerate}
In particular no optimal trajectory whatever --- tame or not --- reaches a point of $\Sigma$ transversally; that part of the statement uses only Lemma~\ref{lem:normalsign}.
\end{proposition}

\begin{proof}
The identity $|\nabla_{\!\C}(u|_\C)|=|\taugrad u|=1$ on $\Sigma$ is the definition of $\Sigma$ together with $\nabla_{\!\C}(u|_\C)=\taugrad u$.

Let $y\in\Sigma$ and let $\gamma$ be optimal to $y$. By Lemma~\ref{lem:normalsign}, $|\nabla u(y)|=1$ and $\nderiv u(y)\ge0$; with $|\taugrad u(y)|=1$ this forces $\nderiv u(y)=0$, so $\nabla u(y)$ is tangent to $\C$. As in the proof of Proposition~\ref{prop:grazing}, $u(\gamma(t))=t$ gives $\dot\gamma=\nabla u$ at every time of differentiability, so the terminal velocity of $\gamma$ at $y$ is tangent to $\C$: the arrival is never transversal, which is the last assertion.

If the terminal arc lies in the bulk we are in case~(G). Otherwise the contact set accumulates at the terminal time, and by the tameness hypothesis it then contains a nondegenerate interval ending at that time, so the arc lies in $\C$ over such an interval and, by Theorem~\ref{thm:tangential}, with tangential velocity. Without tameness this step fails: tangential velocity at each of a sequence of contact times does not place an interval of the trajectory inside $\C$, and the contact set could in principle be a Cantor set. It cannot meet $\C_{\mathrm{sing}}$ at times accumulating at the terminal time: $\C_{\mathrm{sing}}$ is closed, so such times $t_k\uparrow T$ would give $\gamma(t_k)\to y$ with $\gamma(t_k)\in\C_{\mathrm{sing}}$, forcing $y\in\C_{\mathrm{sing}}$ and contradicting $y\in\Sigma\subset\widehat\C$, whose points lie over $\C_{\mathrm{reg}}$. Hence some terminal sub-interval of contacts lies entirely in $\C_{\mathrm{reg}}$. Hence the terminal arc lies in $\C_{\mathrm{reg}}$ and is an admissible curve within it which, being a sub-arc of a length minimiser, is a unit-speed intrinsic geodesic: case~(S).

The transversality exclusion is proved first below and needs no tameness. We emphasise also that (G) cannot be discarded. A tangentially arriving bulk ray gives $|\taugrad u(y)|=1$, hence $y\in\Sigma$, and nothing proved here excludes such $y$ from the relative interior of $\Sigma$; the dichotomy is the correct statement, and it is exactly the pairing of creeping rays with shadow-boundary rays in the geometric theory of diffraction (Remark~\ref{rem:GTD}).
\end{proof}

We next show that the optimal ray grazes the crack at the free boundary.

\begin{proposition}\label{prop:grazing}
Let $y_0\in\Sigma$; in particular $y_0$ may be a free-boundary point at which $|\taugrad u|=1$, i.e.\ a point of $\Gamma\cap\Sigma$. Then
\[
\nderiv u(y_0)=0 ,
\]
that is, $\nabla u(y_0)$ is tangent to $\C$: every optimal trajectory through $y_0$ meets and leaves the crack tangentially, the incidence angle being $\theta=\pi/2$. If moreover $u|_\C$ is of class $C^1$ near $y_0$, then $|\taugrad u|$ is continuous across $\Gamma$ and equals $1$ there, so that the incidence angle tends to $\pi/2$ as $y\to y_0$ from the illuminated side.
\end{proposition}

\begin{proof}
By Lemma~\ref{lem:normalsign}, $|\nabla u(y_0)|=1$ and $\nderiv u(y_0)\ge0$. Decomposing $\nabla u(y_0)$ into tangential and normal parts,
\[
1=|\nabla u(y_0)|^2=|\taugrad u(y_0)|^2+\bigl(\nderiv u(y_0)\bigr)^2=1+\bigl(\nderiv u(y_0)\bigr)^2 ,
\]
whence $\nderiv u(y_0)=0$ and $\nabla u(y_0)=\taugrad u(y_0)\in T_{y_0}\C$ is a unit tangent vector. It remains to identify this vector with the direction of travel, in both time directions. Let $\gamma$ be any optimal trajectory passing through $y_0$ at time $t_0$, parametrized by arc length from $x^*$; then $u(\gamma(t))=t$ for every $t$, since each sub-arc is optimal. Let $d$ be any limit point of the difference quotients $(\gamma(t_0+h)-\gamma(t_0))/h$ as $h\to0^+$; such limit points exist and satisfy $|d|\le1$, $\gamma$ being $1$-Lipschitz in arc length. From $u(\gamma(t_0+h))-u(\gamma(t_0))=h$ and the differentiability of $u$ at $y_0$ up to the lip we get $d\cdot\nabla u(y_0)=1$, and with $|\nabla u(y_0)|=1$ and $|d|\le1$ Cauchy--Schwarz forces $d=\nabla u(y_0)$. The same argument with $h\to0^-$ gives the arrival direction. Hence both the terminal direction of $\gamma$ at $y_0$ and the initial direction of its continuation exist and equal $\nabla u(y_0)$, hence are tangent to $\C$. This is the grazing statement, for \emph{every} optimal trajectory through $y_0$, and the incidence angle $\theta$ defined by $\nderiv u=\cos\theta$ equals $\pi/2$.

The last assertion is then immediate: on $\Sigma$ one has $|\taugrad u|=1$ by definition, at $y_0$ one has $|\taugrad u(y_0)|=1$, and if $u|_\C\in C^1$ near $y_0$ the function $|\taugrad u|$ is continuous there, so its value $\sin\theta(y)$ at illuminated $y$ tends to $1$.

Observe that no limiting argument over nearby isolated contacts is needed, and in particular the proposition does not require $y_0$ to be an isolated contact point. Observe also that the hypothesis $y_0\in\Gamma$ is not used: the argument applies verbatim at every point of $\Sigma$, and we have stated it at $\Gamma$ only because that is where it is applied.
\end{proof}

Finally, we address the regularity of the free boundary.

\begin{proposition}\label{prop:regularity}
Work in a $C^1$ chart in which a neighborhood $W$ of $y_0$ in $\C$ is the graph of a $C^1$ function over an open set of $\mathbb R^{n-1}$, and transport $\Sigma,\Gamma$ to that chart; write points of the chart as $z+te$ with $e$ a fixed unit vector and $z$ in the hyperplane $e^\perp$. Assume $u|_\C\in C^1(W)$ and set $g:=1-|\taugrad u|$, a continuous nonnegative function on $W$, so that $\Sigma\cap W=\{g=0\}$ is relatively closed in $W$ and $\Gamma\cap W\subset\Sigma$. Assume further:
\begin{enumerate}
\item[\emph{(H1)}] \emph{(subgraph)} if $z+te\in\Sigma\cap W$ and $s<t$ with $z+se\in W$, then $z+se\in\Sigma$;
\item[\emph{(H2)}] \emph{(uniform exit rate)} there are $c>0$, $t_0>0$ and $q\ge1$ with
\begin{equation}\label{eq:nondeg}
g(y+te)\ \ge\ c\,t^{\,q}
\qquad\text{for every } y\in\Gamma\cap W \text{ and every } t\in(0,t_0) \text{ with } y+te\in W .
\end{equation}
\end{enumerate}
Then $\Gamma\cap W$ is the graph of a function $\tau$ over the projection $Z\subset e^\perp$ of $\Sigma\cap W$, and $\tau$ is continuous on $Z$.
\end{proposition}

\begin{proof}
By (H1), for each $z$ the slice $\{t:z+te\in\Sigma\cap W\}$ is a sub-interval of the slice of $W$, so
\[
\tau(z):=\sup\{t: z+te\in\Sigma\cap W\}
\]
is well defined and finite on $Z$. It is the graph of $\tau$ that $\Gamma\cap W$ coincides with, and this uses both hypotheses: a point $y=z+te$ of $\Gamma\cap W$ with $t<\tau(z)$ would, by (H1), have a whole segment $y+se\subset\Sigma$ for all small $s>0$, so $g(y+se)=0$, contradicting \eqref{eq:nondeg}; and $t>\tau(z)$ is excluded because $\Gamma\cap W\subset\Sigma$. Conversely $z+\tau(z)e$ lies in $\Sigma$ by relative closedness and is not relatively interior to it, again by \eqref{eq:nondeg}. (Hypothesis (H1) alone would not suffice here: it permits relative-boundary points lying strictly below the top of their slice, and it is (H2) that rules them out.) We do not claim $Z$ is open; it is the projection of $\Sigma\cap W$ and need not be.

Continuity of $\tau$: upper semicontinuity follows from (H1) and the relative closedness of $\Sigma$ in $W$, since $z_k\to z$ with $\tau(z_k)\to\bar t$ gives $z+\bar te\in\Sigma$, whence $\bar t\le\tau(z)$. For lower semicontinuity, suppose $\tau(z_k)\to\bar t<\tau(z)$ along a subsequence. Fix once and for all a \emph{small} increment
\[
s\ \in\ \bigl(0,\ \min\{t_0,\ \tau(z)-\bar t\}\bigr),
\]
with $t_0$ the threshold of (H2), and put $t_1:=\bar t+s$. Then $t_1<\tau(z)$, so $z+t_1e\in\Sigma$ by (H1) and $g(z+t_1e)=0$. On the other hand $z_k+t_1e=y_k+s_ke$ with $y_k:=z_k+\tau(z_k)e\in\Gamma\cap W$ and $s_k:=t_1-\tau(z_k)\to s$; for $k$ large, $s_k\in(s/2,t_0)$ and $z_k+t_1e\in W$, so \eqref{eq:nondeg} applies at the admissible increment $s_k$ and gives
\[
g(z_k+t_1e)\ \ge\ c\,s_k^{\,q}\ \ge\ \frac{c}{2^q}\,s^{\,q}\ >\ 0 .
\]
Letting $k\to\infty$ and using the continuity of $g$ contradicts $g(z+t_1e)=0$. Note that the increment at which (H2) is invoked is $s_k\to s$, which we are free to take as small as we please; invoking (H2) at $\tfrac12(\tau(z)-\bar t)$ instead would not be legitimate, that quantity being fixed by the sequence rather than by us. Both semicontinuities are proved for sequences inside $Z$, which is all that is asserted.
\end{proof}

Two remarks on the hypotheses. Because $g\ge0$ vanishes identically on $\Sigma$, every point of $\Sigma$ is a minimum of $g$, so $\nabla_{\!\C}g=0$ wherever it exists on $\Sigma$: no nondegeneracy phrased as a non-vanishing gradient can hold at a point of $\Gamma$, the implicit function theorem is unavailable, and no $C^1$ regularity of $\Gamma$ follows. This is why \eqref{eq:nondeg} is one-sided and carries a free exponent $q\ge1$, which is genuinely larger than $1$ in the model case, as the next remark computes. Hypothesis~(H1) is a genuine assumption and not a consequence of~(H2): \eqref{eq:nondeg} forbids immediate re-entry into $\Sigma$ but not re-entry further along the ray.

\begin{remark}[The order of vanishing at the free boundary]\label{rem:semiconcave}
At $\Gamma$ the degeneracy can be computed exactly in the circular-arc example of Section~\ref{sec:freeboundary}. With $x^*$ at distance $D>R$ from the centre, the lit trace is $u|_\C(\varphi)=\sqrt{D^2+R^2-2DR\cos\varphi}=:u$ and $|\taugrad u|=D\sin\varphi/u$, which equals $1$ exactly at $\cos\varphi_T=R/D$, where $u=\sqrt{D^2-R^2}=\ell$. Differentiating,
\[
\frac{d}{d\varphi}\,|\taugrad u|
=\frac{D\bigl(u^2\cos\varphi-DR\sin^2\varphi\bigr)}{u^3},
\]
and substituting $\cos\varphi_T=R/D$, $\sin^2\varphi_T=\ell^2/D^2$, $u=\ell$ makes the bracket vanish. Hence $g=1-|\taugrad u|$ vanishes to \emph{second} order on the illuminated side,
\[
g(\alpha)\ \sim\ \frac{\alpha^2}{2R^2}
\]
in the arc length $\alpha$ from $T$; for $R=1$, $D=2$ this is confirmed numerically, $g/\alpha^2\to0.500$. So $q=2$ in \eqref{eq:nondeg}, and $g$ does not grow linearly at $\Gamma$. The quadratic vanishing is the crack analogue of the square-root behaviour at a shadow boundary in geometric diffraction (see Remark~\ref{rem:GTD}): the free boundary is visible in $g$ only at second order, so no first-order criterion can locate it, and this is what Proposition~\ref{prop:regularity} is built to accommodate.
\end{remark}

\subsection{Reconciliation with the crack condition}

\begin{remark}[The crack condition is a variational inequality]\label{rem:reconcile}Definition~\ref{def:regimes} clarifies the precise scope of the crack line of \eqref{eq:HJVI}. On the sliding set $\Sigma$ the tangential term is active, $|\taugrad u|=1$; by Lemma~\ref{lem:normalsign} one then has $(\nderiv u)^2=|\nabla u|^2-|\taugrad u|^2=0$, so $\nderiv u=0$ and $\max(|\taugrad u|-1,-\nderiv u)=0$ holds with both terms vanishing. On the illuminated set $I$, however, the optimal trajectory arrives transversally and $u$ solves the \emph{bulk} eikonal $|\nabla u|=1$ up to the near side, with $|\taugrad u|=\sin\theta<1$ and $\nderiv u=\cos\theta>0$; there \emph{both} arguments of the maximum are strictly negative, so the pointwise identity $\max(\cdot)=0$ does \emph{not} hold. What survives globally are the two \emph{inequalities}
\[
|\taugrad u|\le1,\qquad \nderiv u\ge0 ,
\]
whose complementarity product $(1-|\taugrad u|)\,\nderiv u$ vanishes on $\Sigma$ and is positive on $I$. We stress that these two inequalities are not by themselves a characterization of $u$ --- any constant satisfies them --- and that a system whose complementarity product is allowed to be positive is not a complementarity system in the usual sense. The characterizing conditions are those of Definitions~\ref{def:sub} and~\ref{def:super}, as Theorem~\ref{thm:main} shows; the max-form equation \eqref{eq:VI} encodes only the binding regime $\overline\Sigma$, and the inequalities above are what is left of it elsewhere. One should add that even on $\Sigma$ the two arguments of the maximum vanish together, by Lemma~\ref{lem:normalsign}, so \eqref{eq:VI} is there a single scalar condition rather than a genuine alternative.

This bears directly on the Bellman inequality \eqref{eq:super-bellman} of Proposition~\ref{prop:metric-implies-bellman}, which uses $A^{\mathrm{arr}}$ and not $A$. At a lip point where $u$ is differentiable up to the lip one may take $\nabla\varphi=\nabla u$ in \eqref{eq:sup-explicit}; illuminated points included, $\nderiv u=\cos\theta>0$ puts us in the first case and the supremum equals $|\nabla u|=1$. What fails on $I$ is only the \emph{max-form} equation, and only as an equation. What would fail if one replaced $A^{\mathrm{arr}}$ by the velocity set $A$ is the supersolution property itself, the supremum then reading $|\taugrad u|=\sin\theta<1$. The arrival set is thus not a presentational choice but the one that makes the state-constraint condition of Soner \cite{Soner1986} correct at a crack point.
\end{remark}

\subsection{A worked example: circular-arc crack}

We illustrate the three regimes on an exactly solvable configuration exhibiting genuine sliding. Let $\C$ be an arc of the circle of radius $R$ centred at $O$ and let the source $x^*$ lie outside the disk, $|x^*-O|>R$. The arc has two free ends, so both of its sides are reachable; the lip relevant to the discussion below is the convex (outer) one, which carries $x^*$, and its far side is the disk interior. The arc does not close, so the interior is not enclosed: it is reached, at greater cost, by rounding one of the two ends, and the inner lip carries its own trace. That lip has $\kappa=-1/R<0$ and so carries no sliding, by Remark~\ref{rem:flat}, which is why the discussion below concerns the outer lip alone.
Figure~\ref{fig:regimes} shows the resulting regimes.

For a target $x$ in the geometric shadow of the arc (separated from $x^*$ by $\C$), the shortest admissible path is the classical wrap-around geodesic: the straight segment from $x^*$ tangent to the circle at a point $T$, followed by an arc of $\C$, followed by a straight segment tangent to the circle toward $x$.
Writing $\ell:=\sqrt{|x^*-O|^2-R^2}$ for the tangent length and $\alpha$ for the arc-length parameter along $\C$ measured from $T$ into the shadow, the value function on the wrapped portion of the crack is
\[
u|_\C(\alpha)=\ell+\alpha,
\qquad\text{so}\qquad
|\taugrad u|=\Bigl|\tfrac{d}{d\alpha}(\ell+\alpha)\Bigr|=1 .
\]
Thus the wrapped arc $\{\alpha>0\}$ lies in the sliding set $\Sigma$, on which $u|_\C$ solves the intrinsic eikonal, in agreement with Proposition~\ref{prop:sliding}; the formula $u|_\C=\ell+\alpha$ is valid as long as wrapping past the \emph{other} tangency point does not become cheaper, beyond which $u|_\C$ is the minimum of the two branches and fails to be differentiable at their meeting point. Since $x^*$ lies outside the disk there are two tangency points, and $\Sigma$ is the union of the two wrapped arcs minus the point at which the two branches of the minimum meet, where $u|_\C$ is not differentiable and which therefore lies outside $\widehat\C$ altogether; $\Gamma$ is the two-point set the arcs bound; by symmetry we describe one of them and write $T$ for the corresponding point of $\Gamma$. Here $u|_\C$ is real-analytic on each side of $T$ and $T$ itself belongs to $\Sigma$, so $\Gamma\cap\Sigma=\Gamma$ in this example. At $T$ (where $\alpha=0$): the incoming ray $x^*T$ is tangent to $\C$ ($\theta=\pi/2$, $|\taugrad u|=1$), as predicted by Proposition~\ref{prop:grazing}; this is case~(G) of Proposition~\ref{prop:sliding}, the wrapped arc beyond $T$ being case~(S). On the illuminated side $\{\alpha<0\}$ the arc is met by direct rays from $x^*$ at incidence angle $\theta(\alpha)<\pi/2$, so $|\taugrad u|=\sin\theta(\alpha)<1$ and the point lies in $I$; $\theta(\alpha)\uparrow\pi/2$ as $\alpha\uparrow0$, and $g=1-\sin\theta$ vanishes \emph{quadratically} there, with $q=2$ in \eqref{eq:nondeg} and constant $1/(2R^2)$ (Remark~\ref{rem:semiconcave}); Proposition~\ref{prop:regularity} then applies and $\Gamma$ reduces, on this arc, to the isolated transition point $T$.

\begin{figure}[ht]
\centering
\resizebox{\ifdim\width>\linewidth\linewidth\else\width\fi}{!}{%
\begin{tikzpicture}[scale=1.15,>=Latex]
  \def\R{1.6}
  \coordinate (O) at (0,0);
  \coordinate (S) at (-3.5,0);
  \coordinate (T)  at (117.2:\R);
  \coordinate (T') at (242.8:\R);
  \coordinate (T2) at (42:\R);
  \coordinate (X)  at (2.728,-0.639);
  \coordinate (P)  at (195:\R);
  \fill[black!7] (O) circle (\R);
  \node[black!55] at (-0.2,-0.62) {\small far side};
  \draw[very thick] (250:\R) arc[start angle=250,end angle=-70,radius=\R];
  \draw[dotted,very thick,black!75] (117.2:\R) arc[start angle=117.2,end angle=242.8,radius=\R];
  \node at (180:{\R+0.42}) {$I$};
  \node at (65:{\R+0.52}) {$\Sigma$ (sliding)};
  \node at (310:{\R+0.46}) {$\Sigma$};
  \node at (20:{\R+0.78}) {$\C$};
  \fill (O) circle (0.02) node[below right] {$O$};
  \fill[white,draw=black] (250:\R) circle (0.055);
  \fill[white,draw=black] (-70:\R) circle (0.055);
  \node at (256:{\R+0.46}) {\small tip};
  \node at (-74:{\R+0.46}) {\small tip};
  \fill (S) circle (0.045) node[left] {$x^*$};
  \draw[->] (S) -- (T) node[pos=0.60,above,sloped] {\small tangent ray};
  \fill (T) circle (0.04) node[above left=-1pt] {$T\in\Gamma$};
  \fill (T') circle (0.04) node[below left=-1pt] {$T'\in\Gamma$};
  \draw[->,dashed,black!55] (S) -- (T');
  \fill (T2) circle (0.035);
  \draw[->] (T2) -- (X) node[right] {$x$ (shadow)};
  \draw[->,gray] (S) -- (P);
  \fill[gray] (P) circle (0.03);
  \node[gray] at (203:{\R+0.78}) {\small $\theta<\tfrac\pi2$};
\end{tikzpicture}%
}
\caption{Regimes on a circular-arc crack $\C$, drawn to scale for $R=1.6$ and $|x^*-O|=3.5$. The crack is an arc of $320^\circ$ with two free tips, so both of its sides are reachable; the figure shows the outer lip. Direct rays from $x^*$ strike the illuminated part $I$ (dotted, from $117.20^\circ$ to $242.80^\circ$) transversally, where $|\taugrad u|<1$. At the two tangency points $T,T'$ the ray grazes ($\theta=\pi/2$); these form the free boundary $\Gamma$, which here is a two-point set. Beyond them the geodesic slides along $\C$ --- the sliding set $\Sigma$, the remainder of the arc, where $|\taugrad u|=1$ --- before launching a tangent ray toward shadow targets $x$.}
\label{fig:regimes}
\end{figure}
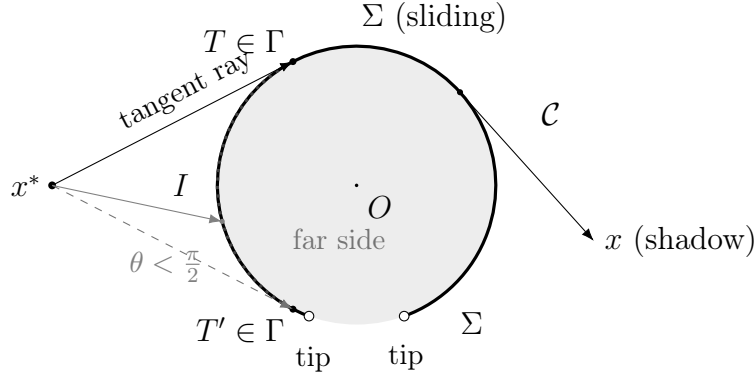

\begin{remark}[Lips of negative curvature carry nothing]\label{rem:flat}
Fix a lip and let $\kappa$ be the curvature of $\C$ taken positive when the curvature vector points to the far side of that lip, as in Theorem~\ref{thm:structure}. If $\kappa<0$ along the lip --- the crack curving strictly away from the lip's own side --- then no optimal trajectory has a sliding arc on that lip. This is not a separate assertion but the contrapositive of Theorem~\ref{thm:structure}(b): a sliding arc would carry $\kappa\ge0$. The flat case $\kappa\equiv0$ is genuinely different and is \emph{not} covered: for a straight crack with $x^*$ on the line containing it, the optimal path to a shadow point runs \emph{along} $\C$ over a nondegenerate interval, so a flat lip can carry a sliding arc. More is true on such a lip: it carries no grazing points either, so $\Sigma$ is empty there outright. Indeed a grazing arrival at $y$ is along a straight bulk segment tangent to $\C$ at $y$ (Theorem~\ref{thm:structure}(a) and Proposition~\ref{prop:grazing}); but for $\kappa(y)<0$ the crack is displaced from its tangent line toward the near side, so that segment would cross it and is not admissible. With both cases (S) and (G) of Proposition~\ref{prop:sliding} excluded, no point of such a lip lies in $\Sigma$. Sliding requires $\kappa\ge0$, the crack bulging \emph{toward} the lip's own side, so that intrinsic geodesics of $\C$ undercut bulk detours; the outer lip of the circular arc is the prototypical such configuration, while its inner lip has $\kappa=-1/R<0$ and carries no sliding. Note that the sign of $\kappa$ is opposite on the two lips of a sheet, so sliding can occur on at most one of them at any given point. 
There is accordingly a threshold in the incidence angle rather than in the curvature alone: for $\kappa>0$ and incidence angle $\theta$, sliding becomes optimal only once $\theta$ reaches $\pi/2$, after which the geodesic creeps along the lip with radius of curvature governed by $\kappa$.
\end{remark}

We close this section by assembling the preceding results into a complete description of optimal trajectories in the plane.

\begin{theorem}\label{thm:structure}
Let $n=2$, let $\Omega$ be convex, and let $\C\subset\Omega$ be a finite union of $C^2$ arcs.
For $y\in\C_{\mathrm{reg}}$, and for the lip of $y$ on which the trajectory lies, let $\kappa(y)$ denote the curvature of $\C$ at $y$ taken positive when the curvature vector points to the far side of that lip.
Let $\gamma:[0,T]\to\overline\Omega$ be an optimal trajectory from $x^*$ to $x$, parametrized by arc length.
Then:
\begin{enumerate}
\item[(a)] on each connected component of the open set $\{t\in(0,T):\gamma(t)\notin\C\}$, $\gamma$ is a straight segment;
\item[(b)] on each nondegenerate interval on which $\gamma(t)\in\C_{\mathrm{reg}}$ (a sliding arc), $\gamma$ is a unit-speed parametrization of an arc of $\C$ along which $\kappa\ge0$; in particular a lip along which $\kappa<0$ carries no sliding arc at all;
\item[(c)] $\gamma$ is of class $C^1$ at every $t$ with $\gamma(t)\in\C_{\mathrm{reg}}$ which is either an endpoint of a sliding arc or an isolated contact; consequently, \emph{for trajectories whose contact set $\{t:\gamma(t)\in\C\}$ is a finite union of points and intervals}, $\gamma$ changes direction only at tips, at junctions, and along sliding arcs. We assume that finiteness rather than deriving it: for a general $C^2$ crack the contact set is merely closed and may be a Cantor set, and we do not know whether an optimal trajectory can realize such a contact set;
\item[(d)] if $\gamma$ changes direction at a tip $p\in\partial\C$ and the sub-arcs adjacent to $p$ are straight segments, with directions $e_{\mathrm{in}}\neq\pm e_{\mathrm{out}}$ --- the excluded reversal $e_{\mathrm{out}}=-e_{\mathrm{in}}$ being impossible for an optimal trajectory --- then the unit tangent $\tau_p$ of $\C$ at $p$ pointing into the crack lies in the closed cone $\{\lambda(-e_{\mathrm{in}})+\mu e_{\mathrm{out}}:\lambda,\mu\ge0\}$, that is, the trajectory bends only around the crack;
\item[(e)] if $\gamma$ leaves a tip $p$ along a segment and has no further contact with $\C$, then $u(z)=u(p)+|z-p|$ at every point $z$ of that segment; in the region swept by such segments the level sets of $u$ are arcs of circles centred at $p$, which thus acts as a secondary source.
\end{enumerate}
\end{theorem}

\begin{proof}
Throughout we use that a sub-arc of an optimal trajectory is optimal between its endpoints whenever these lie off $\C$.

(a) Near a time $t$ with $\gamma(t)\notin\C$, admissibility imposes no constraint, and a locally shortest curve in the convex set $\overline\Omega$ is locally a straight segment.

(b) On a sliding arc the velocity is tangential a.e.\ by Theorem~\ref{thm:tangential}, and $|\dot\gamma|=1$, so $\gamma$ parametrizes an arc of $\C$ by arc length; it does not reverse direction, since an optimal trajectory is injective on such an arc.
Suppose $\kappa<0$ at some point of the arc, hence on a sub-arc $[\gamma(t_1),\gamma(t_2)]$.
There the curvature vector points to the near side, so for $t_2-t_1$ small the chord $[\gamma(t_1),\gamma(t_2)]$ lies on the near side, meets $\C$ only at its endpoints, and is strictly shorter than the arc it subtends; replacing the arc by the chord gives an admissible competitor, contradicting optimality.
Hence $\kappa\ge0$.

(c) At an isolated contact the trajectory is locally a pair of straight segments by (a); these meet $\C$ only at the contact point, so the triangle hypothesis of Corollary~\ref{cor:nocorner} is satisfied by the discussion following that corollary, and the claim follows.
At an endpoint $y=\gamma(t_1)$ of a sliding arc, the adjacent segment arrives (or leaves) with some direction $e$, and the arc continues along the unit tangent $\tau$ of $\C$ at $y$.
If $e\ne\tau$, the trajectory has a corner at $y$ whose inner angle lies on the near side, since $e$ points toward $\C$ from the near side and $\tau$ is tangent; for small $r>0$ the chord joining $y-re$ to the point of the arc at arc length $r$ from $y$ then lies on the near side, meets $\C$ only at its endpoint, and is strictly shorter than $2r$ by the triangle inequality, again contradicting optimality.
Hence $e=\tau$ and $\gamma$ is $C^1$ at $y$.

(d) Suppose $\tau_p$ lies outside the cone $K:=\{\lambda(-e_{\mathrm{in}})+\mu e_{\mathrm{out}}:\lambda,\mu\ge0\}$, and for small $r>0$ put $a:=p-re_{\mathrm{in}}$ and $b:=p+re_{\mathrm{out}}$.
The closed triangle with vertices $a,p,b$ lies in $K+p$, while the arc of $\C$ issuing from $p$ leaves $p$ along $\tau_p\notin K$; since this arc is $C^1$ and the other sheets are at positive distance from $p$, for $r$ small the triangle meets $\C$ only at $p$.
The segment $[a,b]$ therefore lies in $\Omega$, meets no sheet, and is admissible, and $|b-a|<|a-p|+|p-b|$ because $e_{\mathrm{in}}\ne e_{\mathrm{out}}$.
Replacing the sub-arc $a\to p\to b$ by $[a,b]$ shortens $\gamma$, a contradiction.

(e) Let $z$ lie on the final segment $[p,z_0]$.
The sub-arc of $\gamma$ from $x^*$ to $p$ has length at least $u(p)$, so $u(z)\ge u(p)+|z-p|$.
Conversely, an optimal trajectory to $p$ followed by the segment $[p,z]$, which meets $\C$ only at the tip $p$, is admissible, so $u(z)\le u(p)+|z-p|$.
Hence $u(z)=u(p)+|z-p|$, and the level sets of $u$ in the swept region are arcs of circles centred at $p$.
\end{proof}

\subsection{Numerical simulation of the forward problem}\label{sec:numerics}
\label{sec:forwardnum}

We conclude with a simulation that exhibits the three regimes and tests the
statements above quantitatively.
The configuration is the circular-arc crack of Section~\ref{sec:freeboundary}:
$\C$ is an arc of the unit circle centred at the origin spanning $320^\circ$
(so that the two tips leave a gap at the bottom), and the source is
$x^*=(0,2.2)$, outside the disc, so that the outer lip is the lit one.

The admissible distance is computed by a wavefront solver on a visibility graph
whose nodes are $x^*$, the two tips, and $241$ sample points of the crack.
The two edge types implement the two clauses of Definition~\ref{def:admissible}:
a straight edge between two nodes is allowed when the segment does not cross
$\C$ and, at each crack endpoint it has, lies on the side of the lip carried by
that node --- $d\cdot\nC\le0$ for a departure, $d\cdot\nC\ge0$ for an arrival,
by \eqref{eq:A} and \eqref{eq:Aarr}; and a sliding edge of length
$R\,\Delta\theta$ joins consecutive crack samples, realising the tangential
motion permitted along $\C$. The non-crossing clause is thereby enforced as
liftability (Remark~\ref{rem:lift}), so the cavity enclosed by the arc is
reached only by rounding a tip.
Only the outer lip carries graph nodes. Nothing is lost by this: the inner lip
has $\kappa=-1/R<0$, so by Remark~\ref{rem:flat} no optimal trajectory slides on
it, and a point of the cavity is therefore reached by a straight segment from a
tip --- which is what the computation returns, and what
Theorem~\ref{thm:structure}(e) predicts. The companion computation of
Remark~\ref{rem:lips}, where both lips must be reachable because the question
asked there is which of the two is the cheaper, duplicates every sample
instead.

Figure~\ref{fig:forwardlevels} shows the resulting level sets.
They are circles centred at $x^*$ in the lit region, wrap around the crack in
the shadow, and --- inside the cavity --- are arcs of circles centred at the two
tips, which is the secondary-source law of Theorem~\ref{thm:structure}(e).
The illuminated set $I$, the sliding set $\Sigma$, and the free boundary
$\Gamma$ (the two tangency points from $x^*$) are marked on the crack.

Figure~\ref{fig:forwardcheck} tests the theory against the exact solution.
On the crack the computed values agree with the closed form of
Section~\ref{sec:freeboundary} (tangent length plus creeping arc in the shadow,
Euclidean distance in the lit part) to $2.3\times10^{-8}$.
The tangential gradient equals $1$ to four decimals throughout $\Sigma$ and stays
strictly below $1$ on $I$, vanishing at the foot of the perpendicular from $x^*$;
it rises continuously to $1$ at the two points of $\Gamma$, which is the grazing
statement of Proposition~\ref{prop:grazing}.
Finally, at $358$ sample points inside the cavity the identity
$u(x)=u(p)+|x-p|$ of Theorem~\ref{thm:structure}(e) holds exactly, $p$ being the
nearer tip.

\begin{figure}[htbp]
\centering
\includegraphics[width=0.62\linewidth]{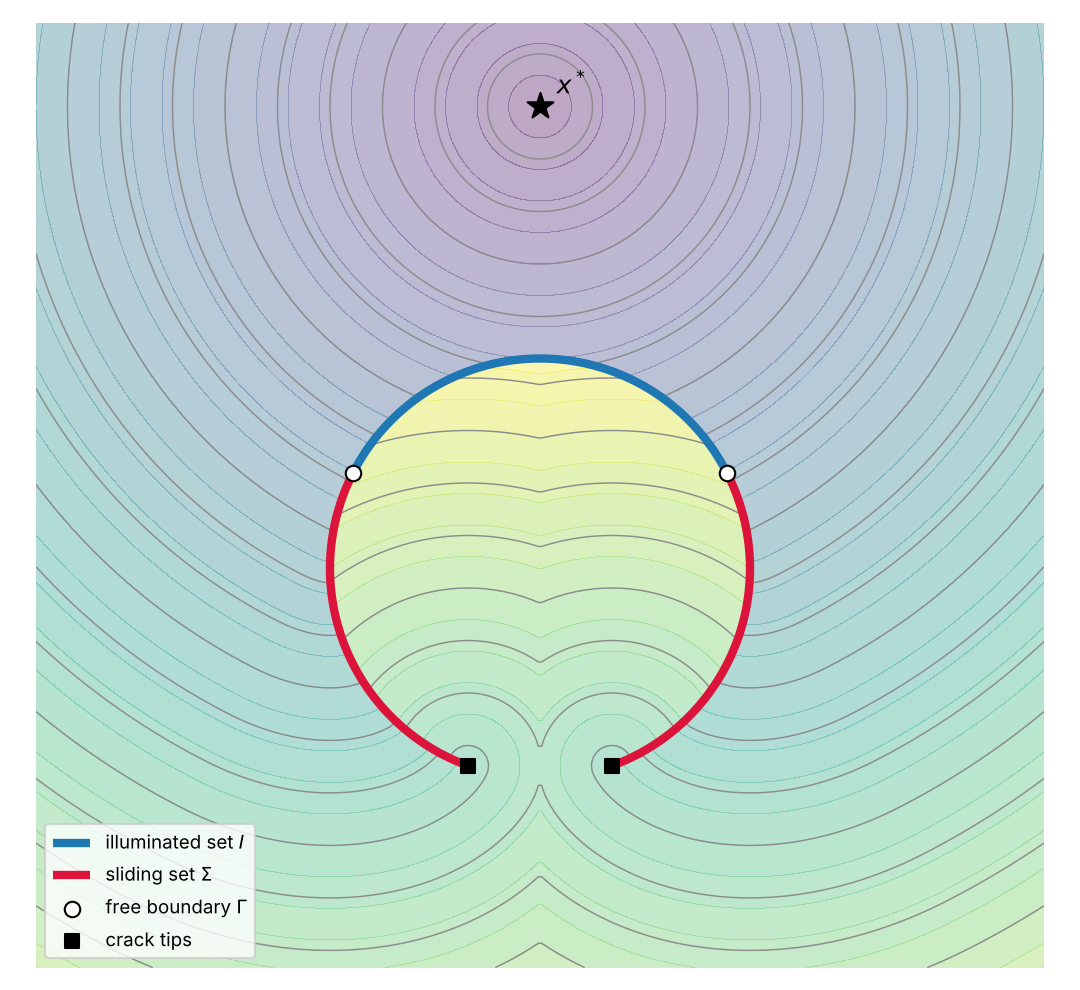}
\caption{Forward simulation for the circular-arc crack with the source $x^*$
outside the disc, so that the outer lip is the lit one. Grey curves are level sets of the
admissible distance $u$: circles centred at $x^*$ in the lit region, wrapping in
the shadow, and arcs of circles centred at the tips inside the cavity. The crack
carries the illuminated set $I$ (blue) and the sliding set $\Sigma$ (red),
separated by the free boundary $\Gamma$ (white circles) at the tangency points.}
\label{fig:forwardlevels}
\end{figure}

\begin{figure*}[htbp]
\centering
\includegraphics[width=\linewidth]{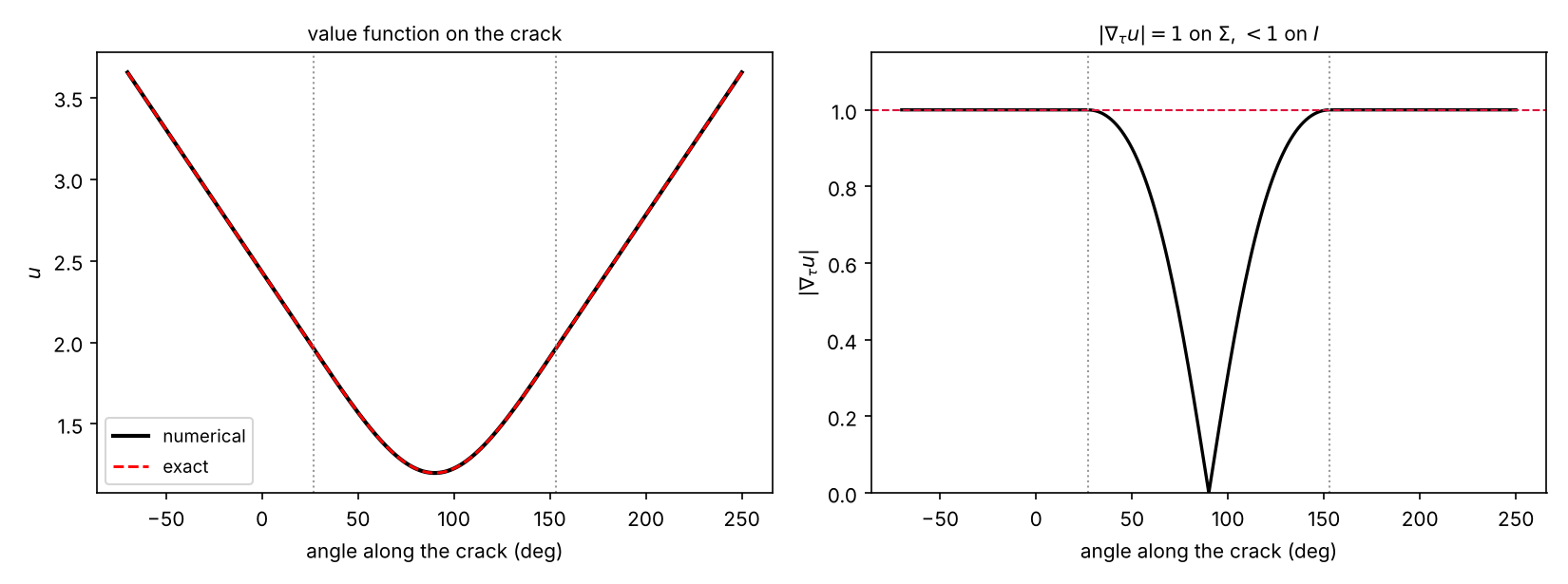}
\caption{Verification on the crack. Left: computed value function against the
closed-form solution (maximum discrepancy $2.3\times10^{-8}$); the dotted lines
mark the free boundary $\Gamma$. Right: the tangential gradient equals $1$ on the
sliding set and is strictly smaller on the illuminated set, rising continuously to
$1$ at $\Gamma$.}
\label{fig:forwardcheck}
\end{figure*}

\begin{remark}[Relation to geometric diffraction]\label{rem:GTD}
The partition $\widehat\C=\Sigma\cup I$, with interface $\Gamma\subset\Sigma$, mirrors the geometric theory of diffraction \cite{Keller1962}: $I$ is the directly illuminated (lit) region, $\Sigma$ carries the creeping (boundary) rays that hug the crack along its near side, $\Gamma$ is the shadow boundary from which diffracted rays are launched tangentially, and the crack tips act as diffraction points seeding the shadow. The present variational derivation reproduces the tangential launching condition of the geometric theory purely from optimality, without a separate diffraction postulate, and Theorem~\ref{thm:structure}(e) makes the role of the tips as secondary sources precise in the plane.
\end{remark}

\section{Conclusion}

We have cast the shortest-path problem in a domain obstructed by a crack as a Hamilton--Jacobi variational inequality on the slit completion of the domain, the crack acting as a two-sided barrier that may be touched and followed but not traversed, with each of its two lips carrying its own state constraint. That the barrier must be modelled as two-sided is not a matter of convenience: Remark~\ref{rem:lips} exhibits a single smooth sheet along which the cheaper lip changes, so no co-orientation of a crack making one face systematically admissible exists in general.

The associated value function is the admissible geodesic distance $\dA$, shown to be the unique viscosity solution of the variational inequality and characterized by a dynamic programming principle together with matching sub- and super-optimality estimates, which combine into a comparison principle (Theorem~\ref{thm:comparison}) sandwiching $\dA$ between every subsolution and every supersolution. Two features of the argument are worth isolating. The supersolution condition is written with the \emph{arrival} set rather than the velocity set, which is what makes the value function a supersolution at illuminated crack points; and the comparison is proved by an inf-convolution in the \emph{intrinsic} metric followed by a descent polygon whose steps are sub-arcs of minimizing geodesics of the slit completion --- so that admissibility of the comparison curve is automatic, and no viability theorem or semiconcavity of the competitor is required.

At the level of trajectories, optimality forces tangential contact at regular crack points --- an optimal trajectory never reflects off the crack --- and in the plane, for a convex domain and a $C^2$ crack, every optimal trajectory is a concatenation of straight segments, sliding arcs, and pivots at tips and junctions, the pivots obeying the one-sided diffraction constraint of Theorem~\ref{thm:structure}(d). On curved cracks each lip splits into an illuminated set $I$ and a sliding set $\Sigma$ with interface $\Gamma$, reproducing from optimality alone the launching geometry of the geometric theory of diffraction; and at $\Gamma$ the defect $1-|\taugrad u|$ vanishes quadratically (Remark~\ref{rem:semiconcave}), so the free boundary is invisible to any first-order criterion. The simulation of Section~\ref{sec:forwardnum} reproduces the three regimes and the tip-diffraction law to the accuracy of the closed-form solution.

Several directions remain open. The free-boundary regularity of Proposition~\ref{prop:regularity} rests on a one-sided, degenerate nondegeneracy condition, and it would be worth knowing whether the quadratic rate found for the circular arc is universal; an affirmative answer would give $\Gamma$ a regularity theory of its own. The structure theorem is stated in the plane, and its extension to $n\ge3$ --- where the tip set is $(n-2)$-dimensional and sliding takes place on a hypersurface --- is open. Finally, the theory yields a natural inverse problem: the admissible geodesic lengths measured between pairs of boundary points determine a crack only through its relative convex hull seen from the boundary, so that the identifiable object is a silhouette rather than the crack itself. We treat that problem, together with the reconstruction schemes it calls for, in a companion paper \cite{CompanionInverse}.


\section*{Acknowledgements}
This research is funded by XXXX under grant number YYYY.
We acknowledge Ho Chi Minh City University of Technology (HCMUT), VNU-HCM for supporting this study.

\end{document}